\documentclass[stropt,onecollarge,envcountsect,envcountsame,runningheads,referee]{svjour2}

\usepackage{mathptmx} 
\usepackage{latexsym}
\usepackage{geometry}
\usepackage{tikz}
\usepackage{soul}
\usetikzlibrary{shapes,arrows}
\usetikzlibrary{positioning,fit}
\usepackage{datetime2}

\usepackage{amsmath}
\usepackage{amsfonts}
\usepackage{color}
\usepackage{xcolor}
\usepackage{amssymb}
\usepackage{graphicx}
\usepackage{soul}
\usepackage[colorlinks]{hyperref}
\usepackage{enumitem}
\usepackage{footmisc}
\RequirePackage[OT1]{fontenc}
\usepackage[utf8]{inputenc}
\usepackage{authblk}
\usepackage{mathtools}
\smartqed

\newcommand{\EXCLUDE}[1]{}

\newcommand{\be}{\begin{equation}}
\newcommand{\ee}{\end{equation}}

\newcommand{\bea}{\begin{eqnarray}}
\newcommand{\no}{\nonumber}

\newcommand{\eea}{\end{eqnarray}}
\newcommand\bmodif{\begin{modif}}
	\newcommand\emodif{\end{modif}}

\renewcommand{\qed}{\hfill $\Box$}
\newcommand{\sP}{\mathbb{P}}  
\newcommand{\sE}{\mathbb {E} } 

\newcommand{\E}{\mathbb {E} }

\newcommand{\EXP}[1]{\mathbb {E}\!\left(#1\right) }
\newcommand{\COV}[1]{\mathsf{Cov}\left( #1 \right)}
\newcommand{\COR}[1]{\mathsf{Corr}\left( #1 \right)}
\newcommand{\Var}{{\rm Var}}

\newcommand{\1}[1]{\mathsf{1}\!\left[\,#1\,\right] }
\newcommand{\remove}[1]{}

\newcommand{\psip}{\psi^!}

\newenvironment{rem}[1][]{\begin{remark}[#1]\rm}{\end{remark}}

\numberwithin{equation}{section}

\def\0{{\bf 0}}

\newcommand{\txi}{ \tilde{\xi}}
\newcommand{\hxi}{ \hat{\xi}}

\def\N{\mathbb{N}}

\def\mR{\mathbb{R}}

\def\mZ{\mathbb{Z}}

\def\mN{\textbf{N}}

\def\mE{\mathbb {E} \,}

\def\X{{\cal X }}

\newcommand{\cN}{{\mathcal N}}
\newcommand{\cK}{{\mathcal K}}

\newcommand{\cL}{\mathcal{L}}

\def\R{\mathbb{R}}
\def\N{\mathbb{N}}

\def\P{\mathcal P}

\def\1{\mathbf{1}}

\def\sumn{\sum^{\neq}}
\def\mt{m_{\top}}
\def\S{{\cal S}}

\newcommand{\red}[1]{\textcolor{red}{\textnormal{#1}}}

\begin{document}

\title{ Central limit theorem for exponentially quasi-local statistics of spin models on Cayley graphs
}

\titlerunning{CLT for statistics of spin models}        

\author{Tulasi Ram Reddy \and  Sreekar Vadlamani \and D. Yogeshwaran \thanks{DY's research was supported by DST-INSPIRE Faculty fellowship and CPDA from the Indian Statistical Institute.}}


\institute{Tulasi Ram Reddy \at
              Division of Sciences, New York University Abu Dhabi \\
              \email{tulasi@nyu.edu}           
        \and
           Sreekar Vadlamani \at
            TIFR Center for Applicable Mathematics, Bangalore, and  \\ Department of Statistics, Lund University\\
            \email{sreekar@tifrbng.res.in}
         \and
         D. Yogeshwaran \at
         Theoretical Statistics and Mathematics Unit,
         Indian Statistical Institute, Bangalore. \\
         \email{d.yogesh@isibang.ac.in}        
}

\date{ }

\maketitle
%
\vspace*{-2.5cm}
\begin{abstract}

Central limit theorems for linear statistics of lattice random fields (including spin models) are usually proven under suitable mixing conditions or quasi-associativity. Many interesting examples of spin models do not satisfy mixing conditions, and on the other hand, it does not seem easy to show central limit theorem for local statistics via quasi-associativity. In this work, we prove general central limit theorems for local statistics and exponentially quasi-local statistics of spin models on discrete Cayley graphs with polynomial growth. Further, we supplement these results by proving similar central limit theorems for random fields on discrete Cayley graphs taking values in a countable space, but under the stronger assumptions of $\alpha$-mixing (for local statistics) and exponential $\alpha$-mixing (for exponentially quasi-local statistics). All our central limit theorems assume a suitable variance lower bound like many others in the literature. We illustrate our general central limit theorem with specific examples of lattice spin models and statistics arising in computational topology, statistical physics and random networks. Examples of clustering spin models include quasi-associated spin models with fast decaying covariances like the off-critical Ising model, level sets of Gaussian random fields with fast decaying covariances like the massive Gaussian free field and determinantal point processes with fast decaying kernels. Examples of local statistics include intrinsic volumes, face counts, component counts of random cubical complexes while exponentially quasi-local statistics include nearest neighbour distances in spin models and Betti numbers of sub-critical random cubical complexes. 

\end{abstract}

	\normalfont
	
\section{Introduction:}
\label{sec:intro}

Our results in full generality are on finitely generated countably infinite Cayley graphs but we shall restrict ourselves to the simple case of integer lattices i.e., $\mZ^d$, to motivate and illustrate our results in the introduction. A simple model of random lattice networks is to consider a random subset of $\mZ^d$ (the $d$-dimensional integer lattice) as nodes and define edges or other features depending on the geometry of the nodes. Under such a set-up, various performance measures/functionals/statistics of the network reduce to geometric functionals/statistics of the underlying set of random nodes. Mathematically, the set of random nodes is a random element in $\{0,1\}^{\mZ^d}$ with $1$ denoting the presence of a node at the corresponding location in $\mZ^d$ and $0$ denoting absence of a node. Denoting the random node set as $\P$, various geometric and topological features of the random network are encapsulated in the random subset (called as {\em cubical complex}) $C(\P) = \cup_{x \in \P} (Q + x)$ with $Q = [-1/2,1/2]^d$ being the centred unit cube and $+$ denoting the Minkowski sum. A common way to understand $C(\P)$ is to investigate the asymptotics of $C(\P_n)$ where $\P_n = \P \cap [-n,n]^d$. Though the mathematical model described above is possibly one of the simpler ones, it appears in various avatars in diverse areas ranging from statistical physics, digital geometry, cosmology, stereology etc. \\


\noindent One of the oft-used model for lattice random networks (cf. \cite{Haenggi2010,Franceschetti2008}) is the percolation model arising from statistical physics (\cite{Grimmett2010}). We note here that percolation models encode only pairwise interaction between nodes and indeed, it is very common for various models of complex networks to be modelled as graphs using pairwise interactions between the nodes (see \cite[Section 6]{Boccaletti2006}). However, in many of these models the interactions are not just pairwise, but happen between subsets of nodes, referred to as higher-order interactions. Hypergraphs represent a natural choice for modelling such higher-order interactions, and so do cubical complexes. The additional topological structure of cubical complexes also makes them suitable for modelling of surfaces or other topological structures as well. A hypergraph can be constructed easily from a cubical complex by making $\{z_1,\ldots,z_k\}$ a $k$-hyperedge if $\cap_{i=1}^k(Q + z_i) \neq \emptyset$. Due to the choice of $Q$, we shall have at most $2d$-hyperedges on the  hypergraph constructed from $C(\P)$. However, it is possible to consider any other compact subset of $\mR^d$ instead of $Q$ to build such hypergraphs, and obtain more general hypergraphs. But for ease of illustrating our main theorems, we shall restrict ourselves to this simple model of cubical complex or the corresponding hypergraph. Both cubical complexes and hypergraphs are very useful models of complex networks. For example, \cite{Atkin74,Atkin76} proposed simplicial complexes to model social networks, and we refer the reader to \cite{Kraetzl01} for a survey on this line of research using simplicial complexes for network analysis. Though simplicial complexes are different from cubical complexes, our methods are applicable to simplicial complexes built on spin models as well.  In \cite{Klamt2009}, hypergraphs have also been proposed as models for cellular networks such as protein interaction, reaction and metabolic networks. Apart from these examples, further examples are discussed in \cite{Estrada2005,Michoel2012} including food webs and collaboration networks as well. Analysis of such networks involve understanding of the corresponding cubical complex or its generalisation (see Section \ref{sec:rcc}). Though we shall not explicitly state applications of our results to such performance metrics, we hope that it will be conspicuous to the reader that our applications in Section \ref{sec:rcc} can be extended naturally to handle other hypergraph statistics as well. \\

\noindent In the community of statistical physicists, where $\P$ is termed as a spin model, there is considerable interest 
in understanding the connectivity properties of spin models (\cite{Copin2015,Grimmett2010,Friedli2017}). This is all the more important when one considers random cubical complexes as models of discrete random surfaces as in \cite{Funaki05,Giacomin2001}. With this perspective, it is natural to study connectivity properties of surfaces. In Sections \ref{para:nng} and \ref{para:toprcc}, we shall mention asymptotics of nearest neighbour distances and Betti numbers of $C(\P_n)$. The latter are a measure of high-dimensional connectivity of surfaces.  \\

\noindent Besides being considered as models of complex networks or discrete surfaces, cubical or cell complexes have often been used in digital image analysis. A simple digital image is nothing but black and white values (i.e., $0,1$-values) assigned to lattice points. Geometry and topology of digital images are of interest in image processing  (\cite{Gray71,Kong1989,Klette2004,Saha2015}), morphology (\cite{Schladitz06}), cosmology and stereology  (\cite{Hilfer2000,Svane2017}). We shall explicitly mention some of the statistics (for example, subgraph counts, component counts) in the Section \ref{para:subcomplex} and these are motivated by those that arise in the above referenced literature. For example, recently Minkowski tensors of random cubical complexes have received attention in astronomy (see \cite{Goring2013,Klatt2016}). Minkowski tensors are a generalisation of intrinsic volumes and asymptotics of the latter shall concern us in Section \ref{para:subcomplex}. \\ 

\noindent It is pertinent to mention here that in \cite[Section 6]{Bulinski12Bernoulli}, it was said about random fields on $\mR^d$ that {\em ``It would be desirable to prove limit theorems for joint distributions of various surface characteristics of different classes of random fields \ldots it might be of interest to prove limit theorems involving other Minkowski functionals for level sets such as the boundary length or the Euler characteristics."} Such a question remains unproven even for random fields parameterised by $\mZ^d$. As level-sets of random fields are an example of spin models (see below), we can safely remark that our general central limit theorems have reduced the task of proving a central limit theorem for many statistics (including those mentioned above) of various spin models on many Cayley graphs to that of proving variance lower bounds. Indeed, this represents one of our major contributions. \\

\noindent We remark here that other statistics of interest arising naturally in signal-to-interference-plus-noise networks (see \cite{Haenggi2010,Franceschetti2008}), $k$-nearest neighbour edges, local topological numbers as in \cite[Section VI]{Saha2015}, discrete Morse critical points as in \cite{Forman02} and local porosity as in \cite{Hilfer2000}, also fall within our framework. \\   

\noindent Even though the afore-cited applications are important contributions of our article as well as crucial motivations for our work, our main contribution can be said to lie within the realm of limit theorems for asymptotically independent random fields as in \cite{Sun91,Doukhan,Bradley05,Heinrich13,Bulinski12,Bjorklund2017}. We shall elaborately survey this literature to contextualise our results. We shall make more specific remarks on the relation between this literature (see Remark \ref{rem:results}) and our contribution after stating our results. \\

\noindent A lattice random field $\X = \{X_z : z \in \mZ^d\}$ is a collection of $\S$-valued random variables ($\S$ is countable) indexed by the lattice points. Various statistics of the random field can be expressed as sums of {\it score functions} $\xi(x,\X)$ which encode the interaction of $X_z$ with $\X$. Setting $\X_n = \{X_z : z \in [-n,n]^d \}$, one is interested in the asymptotics for
\begin{equation}
\label{eqn:Hn}
H^{\xi}_n := \sum_{z \in [-n,n]^d} \xi(z,\X_n).
\end{equation}
Not surprisingly, a large section of literature on random fields is devoted to the simplest of score functions i.e, $\xi(z,\X_n) = X_z$. In what follows, we shall call the corresponding $H^{\xi}_n$ as {\em linear statistic}. In contrast to linear statistic, two other general statistics - {\em local statistics} and {\em exponentially quasi-local statistic} - have received very little attention. The two are defined rigorously in Section \ref{sec:rf-score}. Briefly, $\xi$ is a {\it local statistic} if there exists a (deterministic) $r \in \mN$ such that $\xi(x,\{X_z\}_{z \in x + [-n,n]^d}) = \xi(x,\{X_z\}_{z \in x + [-r,r]^d})$ for all $n > r$. Such statistics are also  $U$-statistics. In vague terms, $\xi$ is {\it exponentially quasi-local} if $r$ is allowed to be random with suitably decaying tail probability. It is to be understood that we refer to the central limit theorem (often abbreviated as CLT) whenever we use the word asymptotics or limit theorems below. \\

\noindent A spin model $\P$ can be naturally viewed as a random field by setting $X_z = \1[z \in \P]$. Under such a consideration, all the aforementioned statistics are either local or exponentially quasi-local statistics. Spin models also arise as {\em level-sets of random fields} i.e., set $\P := \{z \in \mZ^d : X_z \geq u \}$ for some $u \in \mR$. Such level-sets are of interest as well in the applications mentioned above and represent yet another use for our results. \\

\noindent While asymptotics for linear statistics of $\X$ follow naturally when $\X$ is an i.i.d. random field, asymptotics of local or exponentially quasi-local statistics can be deduced from the powerful martingale-difference based central limit theorem of \cite{Penrose01}. For example, we refer the reader to recent application of this general central limit theorem to random topology in \cite{Hiraoka16}.  Asymptotics for local statistics can also be deduced from the results for $U$-statistics of i.i.d. random fields. Once we drop the assumption of independence, the classical methods or even those in \cite{Penrose01} fail. A natural heuristic is that stationary `asymptotically' independent random fields shall display the same asymptotic behaviour as i.i.d. random fields. However it is a challenge to theorise a notion of `asymptotic independence' that is powerful enough to prove various limit theorems and yet accommodative enough to include many interesting examples. To the best of our knowledge, two successful approaches to capture the notion of `asymptotic independence' in random fields are {\em mixing} and {\em $(BL,\Theta)$-dependence}. \\


\noindent In short, mixing conditions require that the `distance' between sigma-algebras of $\{X_z : z \in A\}$ and $\{X_z : z \in B\}$ for $A,B \subset \mZ^d$ decay suitably fast as a function of the distance between $A$ and $B$. We shall mostly focus on $\alpha$-mixing in this article and refer to Section \ref{sec:mixing} for more details. \remove{Different notions of `distance' between sigma-algebras give rise to different notions of mixing though it was shown later in \cite{Bradley93} and \cite[Theorem 2, Section 1.3]{Doukhan} that these mixing conditions are either equivalent to $\alpha$-mixing or finite-range dependence.  We note here that the decay rate mentioned above does not depend on the cardinalities of $A,B$ for {\it stronger} notions of mixing like $\alpha$-mixing, whereas the rate does depend on the cardinalities of $A$ and $B$ for the {\it weaker} notions of mixing like the weak $\alpha$-mixing.} \\ 

\noindent $(BL,\Theta)$-dependence, on the other hand, requires `suitable decay' of covariance of $f(X_z : z \in A)$ and $g(X_z : z \in B)$ for $A,B \subset \mZ^d$ far apart, and bounded Lipschitz functions $f,g$. The `suitable decay' allows for dependence on cardinalities of $A,B$, but in a specific manner, and this indeed is the obstacle in relating mixing to $(BL,\Theta)$-dependence.  We refer to Section \ref{sec:quasi-asso-BL} for details.\remove{This type of covariance decay is reminiscent of proofs involving associated random fields (positively associated, negatively associated or quasi-associated) and indeed $(BL,\Theta)$-dependence is a generalisation of such notions. } \\

\noindent However, we also shall use the weaker notion of clustering arising in statistical physics (cf. \cite{Malyshev75,Martin80}) which roughly states that, for disjoint sets $A,B\subset \mZ^d$
\[ | \sP(A \cup B \subset \P) - \sP(A \subset \P)\sP(B \subset \P) | \]
decays exponentially fast as a function of the distance between $A$ and $B$ with the rate of decay allowed to depend on cardinalities of $A,B$ with a lot more flexibility (see Definition \ref{defn:clustspin}). These different notions of `asymptotic independence' and their relations are summarised in the following figure. 
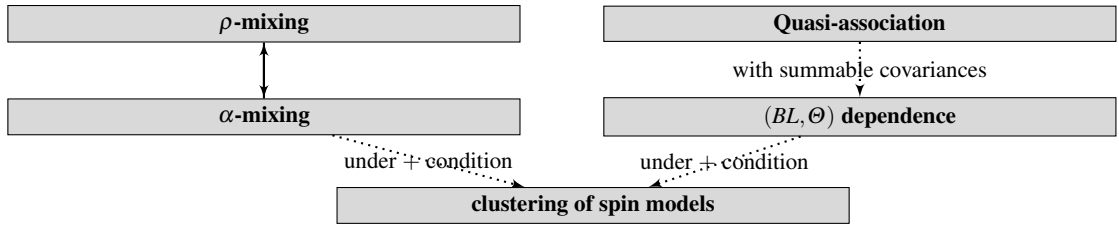
\begin{figure*}[!h]
\begin{center}
\tikzstyle{method} = [draw, rectangle,fill=gray!30, node distance=9ex,
    minimum height=1.5em, text width=0.3\linewidth, text badly centered]
\tikzstyle{assoc} = [draw, rectangle,rounded corners, dashed, fill=gray!30, node distance=9ex, minimum height=2em, text width=0.3\linewidth, text badly centered]
\tikzstyle{line} = [draw, thick, -latex']
\begin{tikzpicture}
   \node [method, text width=0.4\linewidth] (rhomixing)
          {{\bf $\rho$-mixing}};
   \node [method, text width=0.4\linewidth, right=8ex of rhomixing] (qa)
          {{\bf Quasi-association}}; 
   \node [method, text width=0.4\linewidth, below of=rhomixing] (amixing)
          {{\bf $\alpha$-mixing}};
   \node [method, text width=0.4\linewidth,  below right=5ex and -8em of amixing] (clust) 
          {{\bf clustering of spin models}};
   \node [method, text width=0.4\linewidth, below of=qa] (blt) 
          {{\bf $(BL,\Theta)$ dependence}};

  \path [line] (amixing) -- (rhomixing);
  \path [line] (rhomixing) -- (amixing);
  \path [line, dotted] (qa) -- node {with summable covariances} (blt) ;
  \path [line, dotted] (blt) -- node {under {\bf $+$} condition} (clust);
  \path [line, dotted](amixing) -- node {under {\bf $+$} condition}(clust);
\end{tikzpicture}
\caption{\label{fig:mixingrelations} 
Relationships between mixing, quasi-associativity, $(BL,\Theta)$ and clustering. Here {\em $+$ condition} stands for the assumption that the random field is a spin model and coefficient is fast decaying as in the Definition \ref{defn:clustspin}. See Section 
\ref{sec:examples} for precise statements and proofs of these relations.}
\end{center}
\end{figure*}

\noindent While asymptotics for linear statistics are proven under all three conditions - $\alpha$-mixing (\cite{Bradley93}), $(BL,\Theta)$-dependence (\cite{Bulinski12}) and clustering (\cite{Malyshev75,Martin80}) - the non-degeneracy of the limit requires additional assumptions. However, the limit theorems as stated in any of these papers do not apply to local or exponentially quasi-local statistics.  \\

\noindent Thus, in order to use $\alpha$-mixing or $(BL,\Theta)$-dependence or clustering to prove a CLT for $H_n^{\xi}$, one needs to show that $\alpha$-mixing or $(BL,\Theta)$-dependence or clustering holds for the random field $\{\xi(z,\X_n)\}_{z \in W_n}$ under appropriate assumptions on the random field $\X$ and the score functions $\xi$. If we consider $\xi$ generating local statistics, such an approach can be carried out successfully for $\alpha$-mixing random field $\X$ (see Theorem \ref{thm:clt-alpha-local}). The $(BL,\Theta)$-dependence argument does not apply to local statistics i.e., the random field $\{\xi(x,\X_n)\}$ need not be $(BL,\Theta)$-dependent even if $\X$ is so. If we consider $\xi$ generating exponentially quasi-local statistics, assuming exponential $\alpha$-mixing of $\X$ and using ideas involving clustering, we prove a central limit theorem (see Theorem \ref{thm:clt-alpha-quasi-local}). \\ 

\noindent Further, even when these two notions have been used to prove central limit theorems, the attention has been restricted to either random fields on $\mZ^d$ or $\mR^d$ but rarely on more general spaces. The notions of asymptotic independence and exponentially quasi-local can be naturally defined on any metric space and so it is obvious to wonder whether such extensions to more general spaces and their consequences have been considered. While the ergodic theorem (or strong law in our language) has been studied extensively in the ergodic theory or dynamical systems literature, the extension of central limit theorem to more general spaces (mainly to `nice' groups or their Cayley graphs) has received some attention but still many questions persist. For example, we refer the reader to the survey of \cite{Derriennic2006}, and to the recent articles 
\cite[Theorems 3.4 and 3.5]{Cohen2016}, \cite[Theorems 1.1 and 1.5]{Bjorklund2017}. Though the language is vastly different, the crux of these results is again that the CLT holds for $H_n^{\xi}$ under suitable weak mixing conditions (which are more closer to our clustering condition) on the random fields $\{\xi(x,\P_n)\}_{z \in W_n}$ where $W_n$ is now a ball of radius $n$ centred at the identity element in a group or its Cayley graph. We note here that CLTs have been established on general metric measure spaces for statistics of Poisson point processes (see \cite{Lachieze2017}), however to the best of our knowledge, we are not aware of any work under $\alpha$-mixing or $(BL,\Theta)$ assumptions.\\

\noindent While all the aforementioned results are important precursors to our article, to the best of our knowledge, we are not aware of a CLT with easy-to-use geometric conditions on $\xi$, and simple mixing conditions on $\X$. Such conditions are more common in the literature on point processes in continuous settings such as Euclidean spaces or some nice compact manifolds (see \cite{Yukich2013,Penrose2013,Yogesh16,Lachieze2017}). In this article, we prove similar generic central limit theorems for `nice' geometric statistics of the form \eqref{eqn:Hn} for suitably mixing or clustering random fields on Cayley graphs of finitely generated infinite groups. We make extensive remarks following our main theorems (see Remark \ref{rem:results}) comparing our results with those in literature, and also indicating that the ``{\em problem of establishing sufficient conditions on a function defined on an ergodic dynamical system in order that the central limit theorem holds}" (see Derriennic - \cite{Derriennic2006}) is far from being answered satisfactorily. \\

\noindent Given the above context, \remove{the advantage of clustering condition becomes conspicuous as it is applicable for deriving asymptotics for exponentially quasi-local statistics as well as admits various examples including those satisfying mixing condition or $(BL,\Theta)$-dependence condition such as massive Gaussian free field, off-critical Ising model, determinantal random fields, et al. } we remark that while many of the interesting models (for example, massive Gaussian free field, off-critical Ising model, determinantal random fields, et al.) satisfy the $(BL,\Theta)$-dependence condition, not many satisfy $\alpha$-mixing. In other words, there is a trade-off between $(BL,\Theta)$-dependence and $\alpha$-mixing in that the former admits more examples but the latter is powerful enough in proving asymptotics. To our advantage, clustering condition manages to retain the benefits of both $\alpha$-mixing and $(BL,\Theta)$-dependence. We summarise our central limit theorem in relation to those in the literature in the following table. 
\begin{table}[h] \label{table:CLT}
\centering
\begin{tabular}{|r|c|c|c|}
 \hline  & {\bf Linear statistics} & {\bf Local statistics}  & {\bf Quasi-local statistics} \\ 
 &  &  & \\
\hline {\bf Clustering } & \cite{Malyshev75} {\bf *+} & Theorem \ref{thm:main} {\bf *+} & Theorem \ref{thm:main} {\bf *+} \\ 
\hline {\bf $(BL,\Theta)$} & \cite{Bulinski12} {\bf *} & Theorem \ref{thm:main} {\bf *+} & Theorem \ref{thm:main} {\bf *+} \\ 
\hline {\bf $\alpha$-mixing} & \cite{Bradley15} {\bf *} & Theorem \ref{thm:clt-alpha-local} {\bf *} & Theorem \ref{thm:clt-alpha-quasi-local}{\bf *} \\ 
\hline
\end{tabular}
\centering\caption{Central Limit Theorems : The above table summarises the central limit theorems proved in this paper, and its precursors in the literature. Here, the $+$ refers to the $+$ condition as in Figure \ref{fig:mixingrelations} and $*$ refers to assumption on suitable variance lower bounds. $*+$ means that both $*$ and $+$ conditions hold. We are not explicitly referring to other moment bounds as these can be shown to hold in many of our examples.}
\end{table}

\noindent The direct precursor to this article is \cite{Yogesh16}, where asymptotics for geometric statistics of clustering point processes in $\mR^d$ were proven. While the multivariate central limit theorem (Theorem \ref{thm:multimain}) and $\alpha$-mixing central limit theorems (Theorems \ref{thm:clt-alpha-local} and \ref{thm:clt-alpha-quasi-local}) are new, analogues of our Theorems \ref{thm:clustering_rand_measure} - \ref{thm:main} in $\mR^d$ are known (see \cite[Theorems 1.11 - 1.14]{Yogesh16}). While one expects CLT results for point processes in $\mR^d$ to hold for $\mZ^d$, for more general spaces it is apriori not obvious the inter-relation between growth and amenability of the spaces on one hand and correlation decay rates of the spin models and quasi-locality of the statistics on the other hand. This paper generalises \cite{Yogesh16} in this different direction by giving sufficient conditions for the aforementioned inter-relations at least on discrete spaces. This is a step in the direction of CLTs of geometric statistics of mixing random fields on more general spaces. As mentioned in Remark \ref{rem:results}, our work indicates a way to unify the two frameworks as well as extend similar generic limit theorems to more general spaces.  As with many other papers using clustering condition (see, for instance, \cite{Malyshev75,Martin80,Baryshnikov2005,Nazarov12,Bjorklund2017}), we also derive suitable bounds on mixed moments of the random field $\{\xi(z,\P_n)\}_{z \in W_n}$ and then use the cumulant method to prove the central limit theorem. The idea of bounding mixed moments via factorial moment expansion of \cite{Bartek95,Bartek97} is borrowed from \cite{Yogesh16}. However, the statements and proofs of theorems here are simpler exploiting the structure of discrete Cayley graphs, and also lead to applications with minimal assumptions (see Section \ref{sec:rcc}). This also increases the scope of our applications. For example, the application to $k$-nearest neighbour graphs in the continuum set-up (see \cite[Section 2.3.4]{Yogesh16}) needs the assumption of determinantal point processes while our application here (Theorem \ref{thm:cltnnd}) works for any clustering spin model due to Lemma \ref{lem:voidprob}. Also, our Lemma \ref{lem:voidprob} will be very useful in verifying exponentially quasi-locality for many statistics of clustering spin models. One distinct advantage of discrete point processes over their continuum versions is that usage of Palm theory can be avoided completely. This combined with discreteness of the space yields very easy-to-verify moment conditions compared to the $\mR^d$ case.  For example, local statistics on discrete spaces will be U-statistics and hence trivially satisfy the moment conditions unlike their continuum counterparts. Comparing our applications in Section \ref{sec:rcc} with those of \cite[Section 2.3]{Yogesh16} will emphasize this point. Further, we are able to compare clustering with mixing conditions more directly on discrete Cayley graphs as well as furnish central limit theorems for exponentially quasi-local statistics of exponentially $\alpha$-mixing random fields.

\paragraph{Organisation of the paper :}  We shall first introduce all the preliminaries -- Cayley graphs, random fields, exponentially quasi-local statistics, clustering spin models, mixing coefficients, quasi-association and $(BL,\Theta)$-dependence --  in Sections \ref{sec:prelims} - \ref{sec:quasi-asso-BL}. After introducing the necessary notions, we shall state all our results in Section \ref{sec:clust_results} with a detailed discussion of our results in Remark \ref{rem:results}. In Section \ref{sec:examples}, we shall primarily discuss examples and applications of our results on clustering spin models. In addition, we shall also elaborate on the connections between different notions of `asymptotic independence'. Examples of spin models satisfying our assumptions are mentioned in Section \ref{sec:spinmodels}, and applications of our general results to random cubical complexes are described in detail in Section \ref{sec:rcc}. All our proofs are in Section \ref{sec:Proofs} including the crucial factorial moment expansion in Section \ref{sec:FME}. 
Finally, we conclude with a void probability bound needed in our applications in the Section \ref{sec:voidprob}.

\subsection{\bf Preliminary notions and notation}
\label{sec:prelims}


\paragraph{Cayley graphs :} We shall briefly define the necessary notions related to Cayley graphs here and we point the reader to the rich literature \cite{delaharpe2000,Lyons16,Pete2017} for more details.  Let $(V,+)$ be a countably infinite group with a finite symmetric set of generators $S$ i.e., $h \in S$ iff $-h \in S$, where $-h$ denotes the inverse of $h$ in $V$. By calling $S$ as generators, we assume that $V$ is the smallest subgroup containing $S$. Further, denoting the identity by $O$, we shall assume that $O \notin S.$ The {\em Cayley graph} on $(V,+)$ is the graph $G = (V,E)$ whose vertex set is $V$ and $(g_1,g_2)$ is an edge if $-g_1+g_2 \in S$. In other words, for all $g \in G$ and $h \in S$,  $(g,g+h) \in E.$ By symmetry of $S$, it is easy to see that $(g,h) \in E$ iff $(h,g) \in E$ and so $G$ is an undirected graph. Since $O \notin S$ and $S$ is the generator of $V$, $G$ is a simple, connected, regular graph with vertex degree $|S|$. For any $h \in V$, the group operation $g \mapsto g + h$ is called as {\em translation} of $G$. We emphasise that we have not assumed commutativity of the group, thus $g + h$ need not be same as $h +g $ for all $g,h$. We shall now make the blanket assumption that all our Cayley graphs are countably infinite but finitely generated. We shall call such Cayley graphs as {\em discrete Cayley graphs} in the rest of the article. \\

\noindent We shall use $d := d_G$ to denote the usual graph distance. Observe that $d(x,y) = d(O,-x+y)$ which we shall also denote as $|-x+y|$. For a set $A \subset V$, let $|A|$ denote its cardinality, $W_r(y) := \{ x : d(y,x) \leq r \}$ denote the ball of radius $r$ centred at $x$ and $W_r := W_r(O).$ Since the underlying graph is a Cayley graph, $W_r(y) \cong W_r(x)$ for all $x,y \in V$ where $\cong$ denotes the graph isomorphism. Setting $w_r := |W_r|$, there are two trivial bounds for growth of $w_n$ for discrete Cayley graphs. In particular, there exists $d_1 \geq 1$ and a constant $D >0$ such that for all $n \geq 1$, 
\begin{equation}
\label{eqn:wngrowth}
Dn^{d_1} \leq w_n \leq |S| (|S|-1)^n.
\end{equation}
The lower bound is a result of $w_n$ being a strictly increasing $\N$-valued function (See \cite[Section VI.A.3]{delaharpe2000}), and the upper bound follows by comparing with a regular tree of vertex degree $|S|$.  For a subset of vertices, $A \subset V$, we define its inner vertex boundary as $\partial A := \{ x \in A : \mbox{there exists $y \notin A$ such that $(x,y) \in E$} \}.$ 
\begin{definition}[b-Amenable Cayley graphs]
\label{def:amen}
Let $G$ be a discrete Cayley graph as defined above. We say that a Cayley graph is {\em b-amenable} (amenable with respect to the sequence of balls) if 
\[ \lim_{n \to \infty} \frac{|\partial W_n|}{|W_n|} \to 0.\]
We say that a Cayley graph has {\em polynomial growth} if for some $d_2 > 0$, we have that $|W_n| = O(n^{d_2})$\footnote{Here, we have used the Big O notation of Bachmann--Landau.}.
\end{definition}
Polynomial growth is stronger than $b$-amenability, which in turn is stronger than amenability (see \cite[VII.34]{delaharpe2000}). $b$-amenability is equivalent to the condition that $|W_n|/|W_{n-1}| \to 1$. On a discrete Cayley graph, it is easy to note that $|W_{n+m}| \leq |W_n| |W_m|$ and hence by Fekete's sub-additive lemma $\lim_{n \to \infty} |W_n|^{1/n} = w_G \in [1,\infty)$ exists. Further, if $w_G > 1$ (in particular, includes discrete Cayley graphs with exponential growth) then $G$ is not $b$-amenable (see \cite[VII.34]{delaharpe2000}). We discuss the use of $b$-amenability assumption in Remark \ref{rem:results}[(8)]. \\

\noindent A large class of examples of discrete Cayley graphs with polynomial growth can be constructed via Gromov's famous characterisation of groups with polynomial growth (\cite{Gromov1981} and see also \cite[Theorem 10.1]{Pete2017}). We shall quickly mention this result. For two subgroups $A,B$, define $[A,B]$ to be the {\em commutator subgroup} generated by all $a^{-1}b^{-1}ab$ for $a \in A, b \in B$. A group $V$ is said to be {\em nilpotent} if the descending series of subgroups (called as {\em lower central series}) defined by $V_{n+1} = [V_n,V]$ terminates in the trivial subgroup $\{O\}$. With this quick definition, Gromov's theorem can be stated as follows : A finitely generated group has polynomial growth if and only if it is virtually nilpotent (i.e., contains a nilpotent subgroup of finite index). From the definitions, it is easy to see that abelian groups are nilpotent (since $V_1 = \{O\}$) and hence any group containing an abelian subgroup of finite index is virtually nilpotent. The class of abelian groups includes not only our motivating example of integer lattices $\mZ^d, d \geq 1$ but other lattices on the Euclidean plane as well. Further, for a virtually nilpotent group, it was shown in \cite[Theorem 51]{Pansu1983} that there exists a $d \in \mathbb{N}$ (called the degree of polynomial growth) such that
\begin{equation}
\label{eqn:exactpoly}
\lim_{n \to \infty} n^{-d}w_n \in (0,\infty).
\end{equation}
Without defining the group, we mention that discrete Heisenberg groups of all dimensions (i.e., $H_n(\mZ)$ over the ring $\mZ$, see \cite[Section 4.1]{Pete2017}) is also an example of a nilpotent group and its degree of polynomial growth is $4$. For more examples refer to the previous references, especially \cite[Chapter 3]{Roe2003}, \cite[Chapters VI and VII]{delaharpe2000}.

\remove{ 
Let $O$ be the origin of $\mZ^d$, and $L_{\infty}$-norm be the usual supremum norm on $\mZ^d$ denoted by 
$| \cdot |$. At the risk of inducing some confusion, we shall also use $| \cdot |$ to denote the $\mZ^d$-volume of a set i.e., $|A| = \sharp (A \cap \mZ^d)$. Further, let $W_r(x) = [-r,r]^d$ be the $L_{\infty}$-ball of radius $r$ in $\mZ^d$. Often, we shall drop the prefix $L_{\infty}$. We shall abbreviate $W_n(O)$ by $W_n$, and denote its cardinality by $w_n$, which equals $(2n+1)^d$.  We denote the boundary of $W_r$ by $\partial W_r := W_r \setminus W_{r-1}$, and in general, the boundary of a set $S \subset \mZ^d$ by $\partial S := \{x \in S : W_1(x) \nsubseteq S \}$. As is to be expected the choice of $L_{\infty}$ norm is primarily for the sake of simplicity of some of our applications to random cubical complexes, and the results can be rephrased under any other equivalent metric (for example, $L_p$ metric for any $p \geq 1$) as well. }
 
 \subsection{\bf Random fields and score functions}
\label{sec:rf-score}
Let $\X := \{X_x\}_{x \in V}$, be an $\S$-valued random field, where $\S$ is a countable set. 
We say that the random field $\X$ is {\em stationary} if $\{X_x\}_{x \in V} \stackrel{d}{=} \{X_{x+y}\}_{x \in V}$ for all $y \in V$. The configuration space of such a random field is $\Omega := \S^{V}$, which is equipped with its canonical 
$\sigma$-algebra.\\
 
\noindent We shall now introduce score functions which are defined on the configuration space. Specifically, a {\em score function} $\xi$ is a function $\xi : V \times \Omega \to \mR$, which is measurable with respect to product $\sigma$-algebra on $V\times \Omega$. We shall assume all our score functions $\xi$ to be {\em translation-invariant} i.e., for all $y \in V$, 
$\xi(\cdot+y, \vartheta_y(\cdot)) = \xi(\cdot ,\cdot)$, where $\vartheta_y:\Omega\to\Omega$ is translation by $y$ 
in the configuration space defined by $(\vartheta_y \omega)_x = \omega_{x+y}$ for all $x\in V$. Next, for $\omega \in \Omega$, set $\omega \cap S := \{\omega_x : x \in S\}$ for $S \subset V$ which is another random field but indexed by $S$.  Also, for any $r \in \N$, we set $\omega_r := \omega \cap W_r$. We say that a score function is {\em local} if there exists $r > 0$ such that for any $\omega \in \Omega$ it holds that 
\begin{equation}
\label{eqn:local}
 \xi(O,\omega) = \xi_0(O,\omega_r),
\end{equation}
for a (measurable) score function $\xi_0 : \S^{W_r} \to \mR$. In statistical physics, the aforementioned local statistics are referred to as local observables (see \cite[Section 2]{Schonmann1994}) or local function (see \cite[Definition 3.11]{Friedli2017}). Further, due to discreteness of the underlying space,  the total mass (\eqref{eqn:Hn}) induced by local score functions are nothing but local $U$-statistics and this is important in simplifying some of our proofs.  \\

\noindent More generally, for a fixed score function $\xi$, not necessary local, we define its {\em radius of stabilisation} 
$R^{\xi}(x,\omega)$ as follows: for any fixed $\omega\in\Omega$ and $r\in \N$, define 
$\Omega_r(\omega) = \{\omega'\in\Omega:\,\, \omega'_r = \omega_r\}$, then
\begin{equation} \label{eqn:radius-stabilization}
R(O,\omega) := R^{\xi}(O,\omega) = \inf\{r >0 :\, \xi(O,\omega) = \xi(O,\omega'), \,\, \forall \,\, \omega'\in\Omega_r(\omega)\},
\end{equation}
and $R(x,\omega) = R(O,\vartheta_{-x}(w))$. Measurability of $R(O,.)$ follows because for $r \in \mN$, we have that
$$ \{w : R(O,w) \leq r \} = \cup_{w \in \Omega} \Omega^{\xi}_r(w),$$
where $\Omega_r^{\xi}(w) = \Omega_r(w)$ if $\xi(O,.) : \Omega_r(w) \to \mR$ is a constant and else $\Omega_r^{\xi}(w) = \emptyset.$ Now, since  
$\Omega_r^{\xi}(w)$ is measurable and there are at most countably many distinct $ \Omega_r(w)$'s, measurability of $R(O,.)$ follows.

We shall adopt the convention that $W_{\infty} = V$ and $\X_n = \X \cap W_n$. Trivially, $\X_{\infty} = \X$. 

\begin{definition}[Stabilising score function]
\label{defn:stabscore}
We say that $\xi$ is {\em a stabilising score function} on $\X$ if there is a constant $A > 0$ and a function $\varphi(t)$ with $\varphi(t) \downarrow 0$ as $t \to \infty,$ such that 
\[ \sup_{1 \leq n \leq \infty} \sup_{x \in W_n} \sP(R(x,\X_n) \geq t) \leq A \varphi(t). \]
\end{definition}
We say that $\xi$ is a {\em exponentially quasi-local score function} on $\X$ if $\xi$ is a stabilising score function as in Definition \ref{defn:stabscore} and  
\begin{equation}
\label{eqn:quasilocal}
\limsup_{t \to \infty} \frac{\log \varphi(t)}{t^c} < 0,
\end{equation}
for some $c \in (0,\infty)$. If $\xi$ is a local score function, then $R(x,\X) \leq r$ a.s., i.e., $\varphi(t) = 0$ for all $t > r$ and hence it is trivially exponentially quasi-local. Our definition \ref{defn:stabscore} is a quantified version of of \cite[Lemma 6.21]{Friedli2017}, where there is no assumption on the rate of decay. Since such functions are termed as {\em quasi-local} in \cite[Lemma 6.21]{Friedli2017}, we feel it is apt to call our functions as {\em exponentially quasi-local}. However, we wish to point out that in stochastic geometry literature, such score functions are also termed as {\em exponentially stabilizing}.   \\

\noindent As is to be expected, we shall need a suitable moment condition as well on the pair $(\xi,\X)$.  We say that $(\xi,\X)$ satisfies the {\em $p$-moment condition} if 
\begin{equation}
\label{eqn:pmom}
\sup_{1 \leq n \leq \infty} \sup_{x \in W_n} \sE( \max \{|\xi(x,\X_n)|, 1\}^p) \leq M_p < \infty,
\end{equation}
where we assume without loss of generality that $M_p$ is non-decreasing in $p$. Though the finiteness above does not depend upon taking maximum, the definition of $M_p$ as above will be notationally convenient in our proofs.  

\paragraph{\bf Spin models:}  A specific class of random fields when $\S = \{0,1\}$ are called {\it spin models}. Let us denote $\cN = \{0,1\}^V$ as the space of {\em spin configurations}\footnote{Usually, lattice spin configurations are defined as elements of $\{-1,+1\}^V$ but this is trivially equivalent to our definition. Further, these spin-configurations are also referred more specifically as two-state spin configurations to emphasise that spins here can take only two values instead of multiple values.} on $V$. Alternatively, we can think of $\mu = \{ \mu_x \}_{x \in V} \in \cN$ as the space of simple point measures in $V$ by setting $\mu(\cdot) =  \sum_{x \in A} \mu_x.$  We can also identify $\mu$ with its support $\{x \in V : \mu_x = 1\}$. We shall use either the measure-theoretic notation $\mu(\cdot)$ or the set-theoretic notation $\mu \subset V$. 
\begin{definition}[Clustering spin models]
\label{defn:clustspin}
By a spin model or point process $\P$ on $V$, we refer to a $\cN$-valued random variable. Further, let $\P$ be stationary (i.e., $\P + x \stackrel{d}{=} \P$ for all $x \in V$) and that $\P$ is non-degenerate (i.e., $\sP(O \in \P) \in (0,1)$). We say that a non-degenerate and stationary $\P$ is a {\em clustering spin model}, if for all $p,q \geq 1$ there exists constants $C_{p+q}, c_{p+q}$ and a fast decreasing function $\phi$ (i.e., $\phi \leq 1$, decreasing and $\lim_{s \to \infty}s^m \phi(s) = 0$ for all $m \geq 1$) such that for all distinct $x_1,\ldots,x_{p+q} \in V$, we have that 
\begin{equation}
\label{eqn:clustering}
|\sP(\{x_1,\ldots,x_{p+q}\} \subset \P) - \sP(\{x_1,\ldots,x_p\} \subset \P)\sP(\{x_{p+1},\ldots,x_{p+q}\} \subset \P)| \leq C_{p+q}\phi(c_{p+q}s).
\end{equation}
where $s = d(\{x_1,\ldots,x_p\},\{x_{p+1},\ldots,x_{p+q}\}) := \min_{1 \leq i \leq p, 1 \leq j \leq q} |-x_i + x_{p+j}|.$ Without loss of generality, we assume that $C_k$ is non-decreasing in $k$ and $c_k$ is non-increasing in $k$. 
\end{definition}
We shall say that $\phi$ is {\em summable in $G$} if for all $c > 0, p \geq 1$,
\begin{equation}
\label{eqn:phisum}
\sum_{x \in V} \phi^{1/p}(c|x|) = \sum_{n \geq 0}|\partial W_n|\phi^{1/p}(cn) < \infty.
\end{equation}
Trivially, compactly supported $\phi$ are summable on any discrete Cayley graph $G$ and fast decreasing $\phi$ are summable on any discrete Cayley graph with polynomial growth. If $\liminf_{n \to \infty}n^{-1-\epsilon}\log \phi(n) < 0$ for some $\epsilon >0$, then $\phi$ is clearly summable on any discrete Cayley graph because of the upper bound on $w_n$ in \eqref{eqn:wngrowth}. 
\begin{remark}
\label{rem:clust}
Though seemingly simple, clustering has various implications for spin models (see Lemma \ref{lem:voidprob}, for instance). Another important consequence is that if $\P$ clusters, so does $\P^c := V \setminus \P,$ where we have identified $\P$ with $supp(\P)$ (see \eqref{eqn:clust0spin} for a proof). If $C_k,c_k$ are the clustering constants and $\phi$ the clustering function for $\P$, then $2^kC_k,c_k$ are the clustering constants and $\phi$ the clustering function for $\P^c$.
\end{remark}
Next, setting the {\em joint density} or {\em $k$-correlation functions} of $\P$ as
\[ \rho^{(k)}(x_1,\ldots,x_k) = \sP(\{x_1,\ldots,x_k\} \subset \P) ,\]
where $x_1,\ldots,x_k$ are assumed to be distinct and else $\rho^{(k)}(x_1,\ldots,x_k) = 0$. \\

\noindent A spin model $\P$ is said to be {\em exponentially clustering} if $\P$ is a clustering spin model as in Definition \ref{defn:clustspin} with $c_k \geq c > 0,$ and $C_k = O(k^{ak})$ for some $a \in [0,1)$, and the clustering function satisfies the growth condition\footnote{We have referred stretched exponential or super-exponential also as exponential for convenience.}
\begin{equation}
\label{eqn:expclust}
\limsup_{t \to \infty} \frac{\log \phi(t)}{t^b} < 0,
\end{equation}
for some $b \in (0,\infty)$. \\

\noindent Since spin models are also random fields, we can define score functions as we did for random fields and all the related definitions carry forward for spin models as well with a few changes. For a spin model, we assume $\xi(x,\mu) = 0$ if $x \notin \mu$. We say that $\xi$ satisfies the {\em exponential growth condition} if for some $\kappa \in [0,\infty)$ and for all $t > 0$ and $\mu \in \cN$(with $\mu_t = \mu \cap W_t$), we have that
\begin{equation}
\label{eqn:powergrowth}
|\xi(O,\mu_t)| \leq e^{c_*t^{\kappa}}.
\end{equation}
\begin{remark}
\label{rem:pmomlocal}
Local score functions (say with $R(\cdot,\P) \leq r$ a.s.) satisfy the growth condition \eqref{eqn:powergrowth} and $p$-moment condition \eqref{eqn:pmom}. To see this, define 
\begin{equation}
\label{eqn:xitbound}
\|\xi_t\|_{\infty} := \sup_{\mu \in \cN}|\xi(O,\mu \cap W_t)| = \sup_{\mu \in \cN(W_t)}|\xi(O,\mu)|, 
\end{equation}
where again $\mu \cap W_t$ denotes the restriction of the spin model to the ball $W_t$ and $\cN(W_t)$ denotes the space of spin configurations on $W_t$. Since $\cN(W_t)$ is finite for all $t > 0$, we trivially have that  $\|\xi_t\|_{\infty} < \infty$ for all $t > 0$. Thus by `locality' of $\xi$, we have that $|\xi(O,\mu)| \leq \|\xi_r\|_{\infty}$ for any $\mu \in \cN$ and implying the $p$-moment condition and the exponential growth condition trivially.
\end{remark}
\begin{remark}
\label{rem:polygrowth}
Let $\xi$ be a exponentially quasi-local statistic as in \eqref{eqn:quasilocal} and satisfy {\em polynomial growth} condition i.e., there exists $\kappa \in [0,\infty)$ such that for all $\mu \in \cN$ and $t > 0$, we have that
%
\begin{equation}
\label{eqn:polygrowth}
|\xi(O,\mu \cap W_t)| \leq c_*t^{\kappa}.  
\end{equation}

Then, using the exponential tail decay of the radius of stabilisation along with the polynomial growth condition, we have that for all $p > 1$,
\[ \sE(\max \{|\xi(x,\P_n)|^p,1\}) \leq \sE(\max \{c^p_*R^{\xi}(x,\P_n)^{p\kappa}, 1\}) < \infty.\]
Thus, we have shown that $\xi$ satisfies the $p$-moment condition \eqref{eqn:pmom} for all $p > 1$ if it is exponentially quasi-local and satisfies the polynomial growth condition. As we shall see, our examples of exponentially quasi-local statistics shall satisfy this weaker condition of polynomial growth and thereby removing the need to prove the $p$-moment condition as well.
\end{remark}

\subsection{\bf Mixing random fields :}
\label{sec:mixing} 

While the notion of clustering suffices for capturing `asymptotic independence' of spin models, we shall need stronger notions to do so for general random fields. Let $\X = \{X_x: x \in V \}$ be a stationary random field as in Section \ref{sec:rf-score}, and let $\sigma_S := \sigma_S(\X) = \sigma (X_s, s \in S)$ for a subset $S \subset V$. We shall often omit the reference to $\X$, when the underlying random field is clear. We define the following two mixing coefficients for subsets $S,T \subset V$  %
$$ \alpha(S,T) :=  \sup_{A \in \sigma_S, B \in \sigma_T} | \sP(A \cap B) - \sP(A) \sP(B)|, \,\,\,\,\,\,\,\,\,\,\,\,\,\,\,\,\,\,\,\,\,\,\,\, 
\rho(S,T) := \sup_{X \in \mathcal{L}_2(S), Y \in \mathcal{L}_2(T)} | \COR{X,Y} |,$$
where $\mathcal{L}_2(S)$ denotes all square-integrable, $\sigma_S$ measurable functions. Further we define 
\[ \alpha(s) := \sup \{ \alpha(S,T) : d(S,T) \geq s \} \, \, ; \, \, \rho(s) := \sup \{ \rho(S,T) : d(S,T) \geq s \}. \]
A random field $\X$ is said to be {\em $\alpha$-mixing} (resp. {\em $\rho$-mixing}) if $\alpha(s) \to 0$ 
(resp. $\rho(s) \to 0$) as $s \to \infty$. Using the above notation and by choosing appropriate indicator functions, one can trivially obtain the following relationship for any fixed $s \in (0,\infty)$, 
\begin{equation}
\label{eqn:mixing-trivial}
\alpha(s) \leq \frac{1}{4}\rho(s).
\end{equation}  

%

%
\noindent In light of \eqref{eqn:mixing-trivial}, it is clear that $\rho$-mixing implies $\alpha$-mixing. However, it is well known that the two mixing properties are equivalent for any dimension. In fact for $d\ge 2$, one can show that
\begin{equation}\label{eqn:mixing}
\rho(s) \le 2\pi \,\alpha(s).
\end{equation}
We refer the reader to \cite[Theorem 1]{Bradley93} for the complete statement and proof when the underlying parameter space for the random field is assumed to be either $\mZ^d$, or $\mR^d$. However, it is not difficult to observe that the arguments set forth in proving \eqref{eqn:mixing} can be used verbatim to conclude the same result  even when the underlying parameter space is assumed to be a discrete Cayley graph as in this paper, with the understanding that generators play the same role as eigen bases of $\mZ^d$ or $\mR^d$.\\

\noindent Standard examples of mixing random fields include stationary Gaussian random fields on $\mZ^d$ with spectral density bounded away from zero (see Section 2.1 \cite[Theorem 2]{Doukhan}). In fact, one can obtain precise decay rates of the mixing coefficients of stationary Gaussian 
random field on $\mZ^d$ if in addition to the condition on the spectral density stated above, the covariance of the random field decays appropriately as a function of the distance between random variables (see Section 2.1 \cite[Corollary 2]{Doukhan}). We refer the reader to \cite{Bradley05} for related discussion and many more examples of processes satisfying various mixing conditions.

\subsection{\bf Quasi-associative and $(BL,\Theta)$}
\label{sec:quasi-asso-BL}

As stated earlier, we shall now introduce $(BL,\Theta)$-dependence and quasi-association to provide a complete account of our results for various dependence structures. 
A square integrable random field $\{X_x:\,x\in V\}$ is {\em quasi associated} if for any disjoint sets $U,W \subset V$, and any bounded Lipschitz functions $f\in \mathbb{R}^{|U|}$ and $g\in \mathbb{R}^{|W|}$, we have
$$|\COV{ f(X_U), g(X_W)}| \le L_f\, L_g \sum_{x\in U; y\in W} |\COV{X_x,X_y}|,$$
where $L_f$ and $L_g$ are Lipschitz constants of $f$ and $g$, respectively. This definition is further refined 
in \cite{Bulinski12} by restricting the class of test functions to bounded Lipschitz functions on appropriate domain, and employing
the Bernstein blocks technique to introduce a sequence $\Theta = \left\{\theta_r\right\}_{r\in\mZ_{+}}$ of nonnegative numbers satisfying
$\theta_r \to 0$ as $r\to\infty$. For a given $\Theta$, a random field $\{X_x,\, x \in V\}$ is said to be {\em $(BL,\Theta)$ dependent} if
for any disjoint sets $I,J\subset V$, and $f\in BL(\R^{|I|})$, $g\in BL(\R^{|J|})$, with $r = d(I,J)$, we have
\begin{equation} \label{eqn:BLtheta}
\COV{ f(X_I), g(X_J)} \le L_f\, L_g \, \min\left(|I|, |J|\right) \theta_r.
\end{equation}

\noindent A stationary quasi-associated random field $\{X_x,\, x \in V\}$ can be tested to be 
$(BL,\Theta)$ by checking the validity of \eqref{eqn:BLtheta} for 
%
%
%
$\theta_r = \sum_{|x| \ge r} \left| \COV { X_0, X_x}\right|$.
Please refer  to \cite{Bulinski12} for more details.

\subsection{\bf Our results}
\label{sec:clust_results}


As mentioned in the introduction, our main statistic of interest is the sum of score functions of a spin model, i.e., our interest lies in the asymptotics of
\[  H^{\xi}_n = H^{\xi}(\P_n) := \sum_{x \in \P_n} \xi(x,\P_n), \]
as $n \to \infty$. We shall now state our abstract theorems for $H^{\xi}_n$ and mention examples of scores and spin models in the next section. For distinct $x_1,\ldots,x_p$ and $k_1,\ldots,k_p \geq 1$, define the { \it mixed moments} of $H_n^{\xi}(\cdot)$ for $1 \leq n \leq \infty$ as
\begin{equation}
\label{def:mixmom}
m^{k_1,\ldots,k_p}_n(x_1,\ldots,x_p) := \sE(\prod_{i=1}^p \xi(x_i,\P_n)^{k_i}).
\end{equation} 
The above quantity is well-defined if $\xi$'s are non-negative or if we assume $\max \{k_1,\ldots,k_p\}$th moment of
$\xi(x,\P_n)$ exists for all $x \in W_n$. In our results the later shall hold and hence we shall not have to bother ourselves about existence of $m^{k_1,\ldots,k_p}_n(x_1,\ldots,x_p).$

These play a crucial role in the proof analogous to that of moments for a usual random variable. The following theorem is the key step in our proofs of the weak law of large numbers and the central limit theorem.
\begin{theorem}[Clustering of mixed moments]
\label{thm:clustering_rand_measure}
Let $G$ be a discrete Cayley graph which together with $(\xi,\P)$ satisfy one of the following two conditions :
\begin{enumerate}
\item $G$ is $b$-amenable; $\P$ is a clustering spin model and 
$\xi$ is a local score function as in \eqref{eqn:local}. 
\item $G$ has polynomial growth; $\P$ is an exponentially clustering spin model as in \eqref{eqn:expclust}, 
$\xi$ is a exponentially quasi-local score function as in \eqref{eqn:quasilocal} satisfying the exponential growth condition \eqref{eqn:powergrowth} and the $p$-moment condition \eqref{eqn:pmom} for all $p \geq 1$. 
\end{enumerate}
Then, the mixed moments of $\{\xi(x,\P_n)\}_{x\in\P_n}$ for $1 \leq n \leq \infty$ satisfies clustering, i.e., there exists constants $\tilde{C}_K, \tilde{c}_K$ such that for all $z_1,\ldots,z_{p+q}$ with  $s =d(\{z_1,\ldots,z_p\},\{z_{p+1},\ldots,z_{p+q}\})$ and $K = \sum_{i=1}^{p+q}k_i, k_i \geq 1$ for $i = 1,\ldots, p+q$,  we have that
\begin{equation}
\label{eqn:clust_rm}
| m^{k_1,\ldots,k_{p+q}}_n(z_1,\ldots,z_{p+q}) -  m^{k_1,\ldots,k_p}_n(z_1,\ldots,z_p) m^{k_{p+1},\ldots,k_{p+q}}_n(z_{p+1},\ldots,z_{p+q}) | \leq \tilde{C}_K \tilde{\phi}(\tilde{c}_K s),
\end{equation}
for a fast decreasing function $\tilde{\phi}$. Further, under Assumption 1 if $\phi$ is summable as in \eqref{eqn:phisum}, so is $\tilde{\phi}$. 
\end{theorem}
We shall first state the weak law of large numbers and then the central limit theorem, for which we recall that $w_n = |W_n|$, where
$W_n$ is a ball of radius $n$ in $G$, defined earlier in Section \ref{sec:prelims}.
\begin{theorem}[Weak law of large numbers for exponentially quasi-local statistics of clustering spin models]
	\label{thm:weakl}
Let $G$ be a discrete Cayley graph which together with $(\xi,\P)$ satisfy one of the following two conditions :
\begin{enumerate}
\item $G$ is $b$-amenable; $\P$ is a clustering spin model with summable $\phi$ as in \eqref{eqn:phisum} and $\xi$ is a local score function as in \eqref{eqn:local}. 
\item $G$ has polynomial growth; $\P$ is an exponentially clustering spin model as in \eqref{eqn:expclust}, $\xi$ is a exponentially quasi-local score function as in \eqref{eqn:quasilocal} satisfying the exponential growth condition \eqref{eqn:powergrowth} and the $p$-moment condition \eqref{eqn:pmom} for some $p > 2$.
\end{enumerate}
Then as $n \to \infty$,
$$w_n^{-1} \Var(H^{\xi}_n) \to \sigma^2(\xi,\P) := \sum_{z \in V} \COV{\xi(O,\P),\xi(z,\P)} \in [0,\infty),$$
and further
$$ w_n^{-1}H^{\xi}_n \stackrel{P}{\to}  \sE(\xi(O,\P)) \in \mathbb{R}.$$
\end{theorem}
The following abstract central limit theorem is an extension of \cite[Theorem 1]{Malyshev75} to general random fields (See also \cite[Remark 1]{Malyshev75}). Since we are unable to find a general statement of the form below, we shall sketch the proof of this later. 
\begin{theorem}[CLT for clustering random fields.]
\label{thm:cltrandfields}
Let $\X_n := \{X_{n,x}\}_{x \in W_n}, n \geq 1$ be a sequence of random fields such that $\sup_{n \geq 1}\sup_{x \in W_n} \sE(|X_{n,x}|^p) < \infty$ for all $p \geq 1.$ Further, we assume that $\X_n$ satisfy \linebreak clustering of mixed moments i.e., there exists constants $C^X_K, c^X_K$ such that for all $z_1,\ldots,z_{p+q}$ with \linebreak $s =d(\{z_1,\ldots,z_p\},\{z_{p+1},\ldots,z_{p+q}\})$ and $K = \sum_{i=1}^{p+q}k_i, k_i \geq 1$ for $i = 1,\ldots, p+q$,  we have that
\begin{equation}
\label{eqn:clust_randfields}
| \sE(\prod_{i=1}^{p+q}X^{k_i}_{n,z_i}) - \sE(\prod_{i=1}^pX^{k_i}_{n,z_i})\sE(\prod_{i=1}^qX^{k_i}_{n,z_i})| \leq C^X_K \phi_X(c^X_Ks),
\end{equation}
where $\phi_X$ is a fast-decreasing function and satisfies the summability condition as in \eqref{eqn:phisum}. Set $H_n := \sum_{x \in W_n}X_{n,x}$. Further, if for some $\nu > 0$, it holds that
\[ \Var(H_n) = \Omega({w_n}^{\nu}), \]
then we have that as $n \to \infty$,
\[ \frac{H_n - \EXP{H_n}}{\sqrt{\Var(H_n)}} \stackrel{d}{\Rightarrow} N(0,1). \]
\end{theorem}
Combining Theorems \ref{thm:clustering_rand_measure} and \ref{thm:cltrandfields}, we obtain easily the following result that is very convenient for use in applications as we shall see in Section \ref{sec:rcc}.
\begin{theorem}[CLT for exponentially quasi-local statistics of clustering spin models]
	\label{thm:main}
Let $G$ be a discrete Cayley graph and together with $(\xi,\P)$ satisfy one of the following two conditions :
\begin{enumerate}
\item $G$ is $b$-amenable; $\P$ is a clustering spin model with summable $\phi$ as in \eqref{eqn:phisum} and $\xi$ is a local score function as in \eqref{eqn:local}. 
\item $G$ has polynomial growth; $\P$ is an exponentially clustering spin model as in \eqref{eqn:expclust}, $\xi$ is a exponentially quasi-local score function as in \eqref{eqn:quasilocal} satisfying the exponential growth condition \eqref{eqn:powergrowth} and the $p$-moment condition \eqref{eqn:pmom} for all $p \geq 1$. 
\end{enumerate}
 Further if for some $\nu > 0$ it holds that
\[ \Var(H_n^{\xi}) = \Omega({w_n}^{\nu}), \]
then we have that as $n \to \infty$,
\[ \frac{H_n^{\xi} - \EXP{H_n^{\xi}}}{\sqrt{\Var(H_n^{\xi})}} \stackrel{d}{\Rightarrow} N(0,1). \]
\end{theorem}
In applications, one is also interested in a joint distributional limit of the vector $(H_n^{\xi_1},\ldots,H_n^{\xi_k})$ where $H_n^{\xi_i}$ is the total mass corresponding to the score function $\xi_i$. We shall combine Theorems \ref{thm:weakl} and \ref{thm:main} along with the Cram\'er-Wold theorem to derive the following multivariate central limit theorem. 
\begin{theorem}[Multivariate CLT for exponentially quasi-local statistics of clustering spin models]
	\label{thm:multimain}
Let $G$ be a discrete Cayley graph and together with $(\xi_i,\P), i = 1,\ldots,k$ satisfy one of the following two conditions :
\begin{enumerate}
\item $G$ is $b$-amenable; $\P$ is a clustering spin model with summable $\phi$ as in \eqref{eqn:phisum} and  
$\xi_i, i =1,\ldots,k$ are all local score functions as in \eqref{eqn:local}. 
\item $G$ has polynomial growth; $\P$ is an exponentially clustering spin model as in \eqref{eqn:expclust}, $\xi_i, i =1,\ldots,k$ are all exponentially quasi-local score functions as in \eqref{eqn:quasilocal} satisfying the exponential growth condition \eqref{eqn:powergrowth} and the $p$-moment condition \eqref{eqn:pmom} for all $p \geq 1$. 
\end{enumerate}
Set $\bar{H}_n := (H_n^{\xi_1},\ldots,H_n^{\xi_k}).$ We have that as $n \to \infty$,
\[ \frac{\bar{H}_n - \EXP{\bar{H}_n}}{\sqrt{w_n}} \stackrel{d}{\Rightarrow} N(0,\Sigma) ,\]
where $\Sigma := (\Sigma(i,j))_{1 \leq i,j \leq k}$ is the covariance matrix defined by
\[  \Sigma(i,j) := \sum_{z \in V} \COV{\xi_i(O,\P),\xi_j(z,\P)} \in \mathbb{R}. \]
\end{theorem}
Though there is no variance lower bound assumption in our multivariate CLT, it is implicit in the result. Note that the limiting Gaussian vector is non-degenerate iff $\sigma^2(\xi_i,\P) > 0$ for some $i \in \{1,\ldots,k\}$. \\

\noindent We shall now move to next class of results which are based on different set of assumptions involving mixing coefficients.
Broadly, the results are the same as stated above in the case of clustering random fields, but with little leeway
allowing more general random fields $\X$. The interest again lies in the asymptotic distribution of appropriately normalised
and scaled sums of $$  H^{\xi}_n = H^{\xi}(\X_n) := \sum_{x \in \X_n} \xi(x,\X_n) .$$
We note here that when the context is clear, we often omit reference to the underlying random field.
\begin{theorem}[CLT for local statistics of $\alpha$-mixing random fields]
\label{thm:clt-alpha-local}
Let $G$ be a discrete Cayley graph with finite symmetric generator, and $\X$ be stationary, $\alpha$-mixing random field 
defined on $G$. Then, writing $\xi$ for a local statistic as defined in \eqref{eqn:local}, we have
as $n \to \infty$,
\[ \frac{H_n^{\xi} - \EXP{H_n^{\xi}}}{\sqrt{w_n}} \stackrel{d}{\Rightarrow} N(0,\sigma^2), \]
where $\sigma^2 = \sum_{z \in G} \COV{\xi(O,\X),\xi(z,\X)}.$
\end{theorem}

\noindent Next, we shall state the analogous limit theorem stated for the sums of exponentially quasi-local score functions
evaluated on random field satisfying certain assumptions on the rate of decay of the mixing coefficients.

\begin{theorem}[CLT for exponentially quasi-local statistics of exponential $\alpha$-mixing random fields]
\label{thm:clt-alpha-quasi-local}
Let $G$ be a discrete Cayley graph and together with $(\xi,\X)$ such that
$G$ has polynomial growth as in Definition \ref{def:amen}, $\X$ is an exponential $\alpha$-mixing spin model (i.e., $\limsup_{s \to \infty} s^{-}\log \alpha(s) < 0$), 
$\xi$ is a exponentially quasi-local score function as in \eqref{eqn:quasilocal} 
and the $p$-moment condition \eqref{eqn:pmom} for all $p \geq 1$. Additionally, let us assume that 
\[ \Var(H_n^{\xi}) = \Omega({w_n}^{\nu}), \]
for some  $\nu>0$.
Then as $n \to \infty$,
\[ \frac{H_n^{\xi} - \EXP{H_n^{\xi}}}{\sqrt{\Var(H_n^{\xi})}} \stackrel{d}{\Rightarrow} N(0,1). \]
\end{theorem}

\noindent A multivariate version of the above theorem can also be concluded using precisely the same
set of arguments as put forth in the proof of Theorem \ref{thm:multimain}.


\begin{rem}[Remarks on our results and future directions]
\label{rem:results}

\begin{enumerate}


\item Before comparing our results specifically with CLTs in the literature on  mixing, $(BL,\Theta)$-dependence or ergodic theory, we wish to comment on the general points of likeness and unlikeness between our results and those. Firstly, our CLT does not require volume-order variance growth unlike the CLTs available in the afore-mentioned literature. Secondly, as mentioned in the introduction, we provide what we believe as easy-to-use geometric conditions on $\xi$ and mixing/clustering conditions on $\P$ for CLTs to hold. All the CLTs, including ours, require a non-trivial variance lower bound assumption for the limit to be non-degenerate. \\

\item The assumption of $\S$ being countable is mainly required to ensure measurability of the radius of stabilization in Theorem \ref{thm:clt-alpha-quasi-local} but the proof of theorem \ref{thm:clt-alpha-local} will work even with $\S$ being a Polish space. \\

\item {\it Comparison with CLTs under mixing conditions :}  CLTs under $\alpha$-mixing condition are known for linear statistics (\cite{Peligrad98}) and here with suitable additional assumptions, we have extended it to local and exponentially quasi-local statistics in Theorems \ref{thm:clt-alpha-local} and \ref{thm:clt-alpha-quasi-local}. But for spin models clustering is a simpler condition to check than $\alpha$-mixing as attested by the many spin models (see Section \ref{sec:examples}) that satisfy clustering condition. \\

\item {\it Comparison with CLTs under $(BL,\Theta)$-dependence :}  Again, CLTs under $(BL,\Theta)$-dependence are proven for linear statistics
(\cite{Bulinski12}) and many spin models do satisfy this condition. But as mentioned before, it is far from clear whether local or exponentially quasi-local statistics of $(BL,\Theta)$-dependent random fields are $(BL,\Theta)$-dependent. This problem arises mainly due to the specific structure of the covariance decay required in the $(BL,\Theta)$ dependence condition. \\

\item {\it Comparison with CLTs on general spaces :} Though the underlying spaces considered in \cite{Derriennic2006,Bjorklund2017} are far more general than ours, the exponential mixing conditions assumed on the random field $\{\xi(z,\P)\}_{z \in V}$ (see \cite[Definition 2.1]{Bjorklund2017}) is similar to our clustering of mixed moments as in \eqref{eqn:clust_rm} and the condition \cite[(1.7)]{Bjorklund2017} plays the role of our summability condition \eqref{eqn:phisum}. Further, for (strictly) exponential clustering (i.e., $b \geq 1$ in \eqref{eqn:expclust}) and Cayley graphs with sub-exponential growth, our summability condition \eqref{eqn:phisum} holds and so does \cite[(1.7)]{Bjorklund2017} (see \cite[Section 3]{Bjorklund2017}). However, we also allow for sub-exponential clustering (see \eqref{eqn:expclust}) as well provided it is suitably fast-decreasing depending on the score function $\xi$ and the growth of the Cayley graph. \\ 

\item {\em Normal approximation :} While our focus has only been on central limit theorems, it is not uncommon to ask for rates of convergence in central limit theorems. The well-known Stein's method has often been used to derive such rates. For example, rates of normal convergence for linear statistics (and also some local and global statistics) for Ising model and some other specific particle systems have been derived recently in \cite{Goldstein2016}. For some models, the clustering property of spin models play a crucial role. \cite{Goldstein2016} exploits the positive association property of spin models and hence applies for increasing statistics of spin models but it is not clear if it applies to clustering spin models or non-linear statistics like in our CLTs. Another way to obtain rates of normal convergence is by obtaining suitable bounds on the growth of cumulants (see \cite[Lemma 4.2]{Grote16} and \cite{Saulis1991}). This method of normal approximation necessitates a more precise quantification of our cumulant bounds. Further, \cite[Lemma 4.2]{Grote16} also gives cumulant bounds needed for moderate deviations. This would be a worthwhile direction to pursue in the future. \\

\item {\em Cumulant Bounds :}  Another use of cumulant bounds to derive CLT is in \cite{Feray2016a,Feray2016} where such bounds are crucially used in the weighted dependency graph method to prove CLTs. In \cite{Feray2016a}, CLTs for local and some global statistics of the Ising model (see Section \ref{sec:ising}) are proved using bounds on cumulants and a generalisation of the dependency graph method. Again, it requires restrictions on regimes for the Ising model and variance lower bounds but it is not obvious if the methods can be adapted to other similar spin models. \\

\item {\em Scaling limits :} Suppose that $V= \mZ^d$, the integer lattice with the generators of the group being $\{ z = (z_1,\ldots,z_d) \in \mZ^d : \|z\|_{1} = 1 \}$ with $\|\cdot\|_{1}$ denoting the $\ell_1$ distance. The corresponding Cayley graph on $\mZ^d$ is nothing but 
\begin{equation}
\label{eqn:cayleyzd}
V = \mZ^d \, \, ; \, \, E = \{(z_1,z_2) : \|z_1 - z_2\|_1 = 1 \}.\end{equation}
Then it is possible to consider a suitably scaled version of the random field $\{\xi(x,\P_n)\}_{x \in W_n}$ and study its scaling limit. Two possible choices for scaling are either to consider the random field \\ $\X_n^1  := \{ (\Var(H_n^{\xi}))^{-1/2}(\xi(n^{-1}x,\P_n) - \sE(\xi(n^{-1}x,\P_n))) \}_{x \in W_n}$ or  the random field \\ $\X_n^2 := \{ (\Var(\sum_{nx \leq y \leq (n+1)x} \xi(x,\P) ))^{-1/2} \sum_{nx \leq y \leq (n+1)x} (\xi(x,\P) - \sE(\xi(x,\P))) \}_{x \in \mZ^d}$ where $\leq$ is the co-ordinate wise ordering of points in $\mZ^d$. The former scaling is more in the spirit of the scaling considered in \cite[(1.3)]{Yogesh16} and the latter is in spirit of the scaling considered in \cite[Theorem 2]{Malyshev75}. In both the cases, the limit is expected to be a Gaussian random field with a white noise like structure i.e., the covariance matrix of all finite dimensional marginals converging to a diagonal matrix (see \cite[(1.21)]{Yogesh16} and \cite[Theorem 2]{Malyshev75}). \\

\item We are not aware of any examples of $b$-amenable Cayley graphs exhibiting non polynomial growth. However, our proof methods for $b$-amenable graphs should allow our results to be proven under the assumption that $w_n$ grows sub-exponentially i.e., $\limsup_{n \to \infty} n^{-1}\log w_n = 0$. For such groups, there exists a subsequence $n_k \to \infty$ as $k \to \infty$ such that $w_{n_k}^{-1}|\partial W_{n_k}| \to 0$ as $k \to \infty$. Thus, if we take sums over $W_{n_k}$ for the subsequence $n_k$ chosen as above instead of taking sums over $W_n$ in \eqref{eqn:Hn}, one would expect our results to hold under such asymptotics as well. \\

\item We have restricted ourselves to a class of amenable Cayley graphs but it is natural to ask whether one can consider a more general class of graphs. Deriving the motivation from various probabilistic studies (see \cite{Aldous2007,Benjamini2013,Lyons16,Pete2017}), two possible classes of graphs that are suitable to such a study are unimodular random graphs and vertex transitive graphs. To further emphasise the need for such a study, even graphs on stationary point processes as studied in \cite{Yogesh16} can be considered as unimodular random graphs (see \cite[Section 5]{Baccelli2016}). Thus a study on unimodular random graphs can unify the framework in this article and that of \cite{Yogesh16} apart from considerably extending the scope of applications. \\

\remove{\item It is very natural to ask if CLT can be proven more general real-valued random fields or at least to more general spin models with finitely many spin values under a suitable clustering condition. As noted in Theorem \ref{thm:cltrandfields}, linear statistics of clustering random fields satisfy a central limit theorem. But it is not obvious if this condition suffices to guarantee central limit theorems for local statistics or exponentially quasi-local statistics of such random fields. We see in Theorem \ref{thm:clt-alpha-local} that under the stronger assumption of $\alpha$-mixing, one can prove a central limit theorem for local statistics, but in order to obtain an analogous result for exponentially quasi-local statistics, we needed a stronger
assumption pertaining to the decay rate of the mixing coefficients (see Theorem \ref{thm:clt-alpha-quasi-local}). }

\item Another persistent but unavoidable assumption not only in our CLTs but in various such generic CLTs in the literature (including those cited here) is the variance lower bound condition. Such lower bounds are usually shown by ad-hoc methods. Primarily, for sums of stationary sequences, variance lower bounds can be obtained under conditions involving the spectral density of the random variables
(see Theorem 2 in Chapter 1.5 of \cite{Doukhan}). Alternatively, under the stationary and strong mixing condition, together with appropriate summability condition of the covariance, the necessary and sufficient condition for meaningful variance lower bounds of partial sums, is that the variance must grow to infinity (see \cite[Lemma 1]{Bradley97}, or \cite[Theorem 2.1]{Peligrad98}). However, for sums of non-stationary sequences of random variables, the condition $\rho'(\X,1) < 1$
provides suitable variance lower bounds (see \cite[Theorem 2.2]{Bradley15}). In this context, a very simple and natural question follows: for an $\alpha$-mixing random field $\X$ 
satisfying $\rho'(\X,1) < 1$, writing $\rho'(\X^{\xi},k)$ as the $\rho$-mixing coefficient of the field $\X^{\xi}$, is it possible to
conclude $\rho'(\X^{\xi},1) < 1$, even for a local statistic $\xi$? 
\end{enumerate}
\end{rem}

\section{Examples and applications}
\label{sec:examples}

In this section, we illustrate our main theorem (Theorem \ref{thm:main}) by providing examples of exponentially quasi-local statistics and clustering spin models. Though we shall mainly focus on a variety of applications to random cubical complexes, we shall also hint at others. Also, we shall not mention applications of our mixing CLTs (Theorems \ref{thm:clt-alpha-local} and \ref{thm:clt-alpha-quasi-local}) but we hope the discussion in the introduction and our applications for clustering spin models will convince the reader that such results are feasible as well.


Prior to discussing the examples, we shall detail relationship between clustering and other measures of 
of association, specifically, mixing and $(BL,\Theta)$, which are going to be used in this section.
Observe that the various mixing coefficients merely extract the dependence structure
without delving into the decay rates of the inherent dependence structure whereas other dependence structures
like the clustering and $(BL,\Theta)$ do, depend and, provide certain information regarding the decay
rates of the inherent dependence structure. Heuristically, appropriate decay assumptions on mixing coefficients
may establish a link from mixing to clustering, and possibly to $(BL,\Theta)$. In this direction, we state the following two propositions establishing connection between mixing and clustering conditions and  between $(BL,\Theta)$ dependence structure and the clustering. 

\begin{proposition}
\label{prop:mixing_clust}
Let $\P$ be a stationary, (strong) $\alpha$-mixing spin random field indexed by $G$, a discrete Cayley graph.
Assume that the mixing coefficient $\alpha(s)$ is a fast decreasing function.
Then, such $\X$ also satisfies the clustering condition with $C_k = c_k = 1$ and $\phi(s) = \alpha(s)$. 
\end{proposition}
%
The proof is rather straightforward, thus we leave it to the reader.
\begin{proposition}
\label{prop:blt_clust}
Let $\P$ be a stationary, spin random field defined on a discrete Cayley graph $G$, and for any vertex $u \in V$ denote $Z_u = \1\{u\in \P\}$
Let the covariance function of $Z$, given by $r(u) = \COV{Z_0,Z_u}$, satisfy the $(BL,\Theta)$  
condition stated in \eqref{eqn:BLtheta}, with $\theta(k) = \sum_{\|u\|\ge k} |r(u)|,$ such that $\theta(k)$ is a fast decreasing function. Then, the random field $Z$ is also a clustering random field with $C_k = k, c_k = 1$ and $\phi(s) = \theta(s)$.
\end{proposition}
\begin{proof}
For any regions $P, Q\subset G$ such that $|P|=p$,
$|Q|=q$,  and $\text{dist}(P,Q) =s$, define $f:\{0,1\}^p\to\R$ and $g:\{0,1\}^q\to\R$ as
$f(t_1,\ldots, t_p) = \prod_{i=1}^p 1_{\{t_i=x_i\}}$ and $g(s_1,\ldots,s_q) = \prod_{j=1}^q 1_{\{s_i=x_{p+i}\}}$
for some fixed $(x_1,\ldots,x_{p+q})$. Next,
setting $P=(u_1,\ldots,u_p)$ and $Q=(v_1,\ldots,v_q)$, clearly, $\E(f(Z_{P})) = \sP(Z_{u_1}=x_1,\ldots, Z_{u_p}=x_p)$ \linebreak
and $\E(g(Z_{Q})) = \sP(Z_{v_1}=x_{p+1},\ldots, Z_{v_q}=x_{p+q})$. Therefore,
\begin{eqnarray*}
|\sP\left(\left(Z_{u_1},\ldots,Z_{u_p},Z_{v_1},\ldots,Z_{v_{q}}\right)=\left(x_1,\ldots,x_{p+q}\right)\right) &-& 
\sP(Z_{u_1}=x_1,\ldots,Z_{u_p}=x_p)\,
\sP(Z_{v_1}=x_{p+1},\ldots, Z_{v_q}=x_{p+q})\big|\\
&=& \left| \COV{f(Z_P),g(Z_Q)}\right|
\leq \left( p \wedge q\right) L_f\, L_g \sum_{\|u\|\ge s} |\COV{Z_0,Z_u}|.
\end{eqnarray*}
Now the result follows easily as $L_f = L_g = 1$. 
\remove{Next, choosing
$C_{p+q} = L_f\, L_g\, 2^{p+q}$ and $c_{p+q}$ as some universal constant, we conclude that
\begin{eqnarray*}
&& \left|\sP(Z_{u_1}=x_1,\ldots,Z_{u_p}=x_p, Z_{v_1}=x_{p+1},\ldots, Z_{v_q}=x_{p+q}) \right.\\
&& \left. - \sP(Z_{u_1}=x_1,\ldots,Z_{u_p}=x_p)\,\sP(Z_{v_1}=x_{p+1},\ldots, Z_{v_q}=x_{p+q})\right| \,\,\,
\le \,\,\,\, C_{p+q} \theta(c_{p+q}s),
\end{eqnarray*}
which is the required clustering condition as stated in \eqref{eqn:clustering}. 
}  \qed
\end{proof}

\subsection{\bf Examples of clustering spin models}
\label{sec:spinmodels}
The simplest example of a clustering spin model is one with i.i.d. spins. The clustering property, which captures asymptotic independence in a strong way, is a natural condition that is expected to hold in statistical physical models which have weak dependences. Without delving into the details, we shall mention a few illustrative examples in this part of the section, and 
specifically restrict our attention to spin models on the lattice $\mathbb{Z}^d$, unless mentioned otherwise.  Notice that by the commutative property of $\mathbb{Z}^d$, the distance $d(x,y) = |-y + x| = |x-y|$, which matches the
$\ell_1$ distance on $\mathbb{Z}^d$.

\subsubsection{Level sets of Gaussian fields}
\label{sec:level_gaussian}
Let $\X = \{X(x)\}_{x \in V}$ be a stationary Gaussian random field whose covariance kernel is exponentially decaying i.e., $\omega(x,y) = \COV{X(x),X(y)}$ is such that $\liminf_{|y| \to \infty} \frac{\log \omega(O,y)}{|y|} < 0.$ Further for simplicity, assume that $\omega(x,y)$ is a function of $|x-y|$ alone. The super-level sets at level $u$ of this Gaussian field, defined as $\P_u := \{x \in V : X(x) \geq u\}$, is a spin model. To show clustering of $\P_u$, we shall use the following total-variation distance bound between Gaussian random vectors from \cite{Beffara16}. \\

\noindent We recall the definition of total variation distance $d_{TV}$ between two probability measures $\mu$ and $\nu$ on a sigma-algebra $\mathcal{F}$ is given by
\[
d_{TV}(\mu,\nu)=\sup\limits_{A\in \mathcal{F}}|\mu(A)-\nu(A)|.
\]
\begin{theorem}[Theorem 4.3 in \cite{Beffara16}]\label{thm:clustering_from_exponential_decay}
	Let $X=(X_1, X_2)$ and $Y=(Y_1, Y_2)$ be random Gaussian vectors (not necessarily centred) with covariances $\Sigma_X=\left(\begin{array}{cc}
	\Sigma_{11} & \Sigma_{12} \\
	\Sigma_{12}^{T} & \Sigma_{22} 
	\end{array}\right)$ and $\Sigma_Y=\left(\begin{array}{cc}
	\Sigma_{11} & 0 \\
	0 & \Sigma_{22} 
	\end{array}\right)$.  Assume that the size of vectors $X_1 $ and $Y_1$ be $m$, while that of $X_2$ and $Y_2$ be $n$. Let $\mu_X$ and $\mu_Y$ be the laws of the vectors whose entries are $1$ or $0$, whether the corresponding entries in the vectors $X$ and $Y$ are positive or not, moreover $\Sigma_{11}$ and $\Sigma_{22}$ has $1$ on diagonal. Then $$d_{TV}(\mu_X,\mu_Y)\leq2^\frac{14}{5}(m+n)^\frac{8}{5}(\max_{i,j}|\Sigma_{12}(i,j)|)^\frac{1}{5}.$$
\end{theorem}
\begin{remark}\label{rem:non-centred_clustering}
	Theorem \ref{thm:clustering_from_exponential_decay} of Beffara and Gayet was originally stated for centred Gaussian vectors in \cite{Beffara16}. By using the trivial bound $\sP(|X|<\epsilon)\leq\epsilon$ for any Gaussian random variable $X$ with arbitrary mean, where ever necessary, in the proof of Theorem \ref{thm:clustering_from_exponential_decay} it trivially extends to any non-centred Gaussian vector. Hence, Theorem \ref{thm:clustering_from_exponential_decay} is true for arbitrary level sets of a Gaussian vector. 
\end{remark}
Thus trivially from Theorem \ref{thm:clustering_from_exponential_decay}, and the corresponding definitions, we have the following corollary.
\begin{corollary}\label{cor:non-centred_clustering}
	Let $\X$ be a stationary Gaussian field, on a discrete Cayley graph $G$, with $X(O)$ having unit variance. The level sets $\mathcal{P}_u$ satisfy exponential clustering with clustering constants $C_k=2^{\frac{14}{5}}k^\frac{8}{5}, c_k =1$ and clustering function $\phi(t) = \omega(O,x)$, where $|x|=t$. 	
\end{corollary}



\paragraph{Massive Gaussian free field :}
Massive Gaussian free field, on the lattice $\mathbb{Z}^d$, is defined as the Gaussian field $G_m(\cdot )$ whose covariance kernel is $g_{m}(x,y)=a_me^{-a'_m|x-y|}$ for all $x,y \in \mathbb{Z}^d$ where $a_m$ and $a'_m$ are some constants depending on $m$ (see  \cite[Proposition 8.30]{Friedli2017}). The function $g_m$ is nothing but the discrete Green's function for the operator $m^2+\Delta$, where $m>0$ is considered as the mass and $\Delta$ is the discrete Laplacian operator on $\mZ^d$.  The super-level set of the massive Gaussian free field at level $u$, $\P_u := \{z:G_m(z)\geq u\}$ is a spin model. Invoking Theorem \ref{thm:clustering_from_exponential_decay} along with Corollary \ref{cor:non-centred_clustering}, $\P_u$ is an exponentially clustering spin model as in \eqref{eqn:expclust} with clustering constants $C_k = 2^{\frac{14}{5}}k^\frac{8}{5},c_k = 1$ and clustering function $\phi(t) = a_me^{-a'_mt}$.

\subsubsection{Ising model} 
\label{sec:ising}
In this example, we consider the simplest form of the Ising model with interactions from the nearest neighbours alone. For a finite set $\Lambda \subset \mZ^d$, the Ising model is defined as the probability distribution on $\Omega_{\Lambda} := \{\pm1\}^{\Lambda}$ such that for $\sigma = \{\sigma_i\}_{i\in\Lambda} \in \Omega_{\Lambda}$, the probability distribution is given by
$$ P(\sigma) = \frac{1}{Z_{\Lambda}} \exp \left\{-\beta \sum_{\substack{i,j\in\Lambda\\|i-j|=1}}\sigma_i\sigma_j-h\sum_{i\in\mathbb{Z}^d}\sigma_i \right\},$$
where $\beta >0$ is the inverse temperature, $h \in \mR$ is the external field, $Z_{\Lambda}$ is a suitable normalising constant, and
the term in the exponential is referred to as {\em interaction potential}. By making the transformation $\sigma_i \rightarrow \frac{\sigma_i+1}{2}$, the Ising model is a spin model in the framework of this article. 
For any $d\geq1$ and $h=0$, there exists $\beta_c(d)$, such that when $0<\beta<\beta_c(d)$ a unique measure on $\mathbb{Z}^d$ is obtained by suitably taking the weak limit as $\Lambda$ approaches $\mathbb{Z}^d$. For any $d\geq1$ and $h\neq0$, the weak limit exists whenever $\beta>0$. For $d=1$, the weak limit exists for any $h \in \mathbb{R}$, for any $\beta \in (0,\infty)$. We refer the reader to \cite{Friedli2017} (Chapters 3 and 6, Theorem 3.25 and Exercise 6.17) for more on the characterisation of these limits. In \cite{lebowitz:penrose:1974} it was established that the Ursell functions decay exponentially, and that the infinite volume limit of these does not depend on the sequence in which the limit is taken. Therefore, combining results from \cite{Friedli2017} and \cite{lebowitz:penrose:1974} for the Ising model on any dimension $d\geq1$ with $0<\beta<\beta_c(d)$, in the absence of external magnetic field (i.e., $h =0$), and for any $\beta > 0$ in the presence of external magnetic field (i.e., $h \neq 0$), the Ursell functions decay exponentially. Ursell functions (mixed cumulants) of a  random field $\X = \{X(x)\}_{x \in V}$ are defined as follows : For $n \geq 1, x_1,\ldots,x_n \in V,$
\[
u_n(x_1,\dots,x_n)=\frac{\partial}{\partial z_1}\dots\frac{\partial}{\partial z_n}\log\mathbb{E}e^{z_1X(x_1)+\dots+z_nX(x_n)}|_{z_1=\dots=z_n=0}.
\]
Joint moments can be expressed using the relation,
\[ \mathbb{E}(X_{x_1}X_{x_2}\dots X_{x_n})=\sum\limits_{\pi}\prod\limits_{B\in \pi}u_{|B|}(X_{x_i}:i\in B), \]
where the sum is over all partitions $\pi$ of $\{1,2,\dots,n\}$ and the product is over all blocks $B$ in the partition $\pi$.
\begin{equation}\label{ursell_clustering1}
\sP(X_{x_1} = X_{x_2} = \dots = X_{x_{p+q}} =1 ) - \sP(X_{x_1}=\dots=X_{x_p}=1) \sP(X_{x_{p+1}}=\dots=X_{x_{p+q}}=1)= \sum\limits_{\pi'}\prod\limits_{B'}u_{|B'|}(X_{x_i}:i\in B'),
\end{equation}
where the sum is over the partitions $\pi'$ of $\{1,2,\dots,p+q\}$, whose blocks $B'$ has at-least one term from each of the sets $\{1,2,\dots,p\}$ and $\{p+1,\dots,p+q\}$. From the exponential decay of Ursell functions, each term in the right hand side of \ref{ursell_clustering1} is decaying exponentially in the distance between the sets $\{x_1,\dots,x_p\}$ and $\{x_{p+1},\dots,x_{p+q}\}.$

\remove{ Therefore we obtain that
\begin{equation}\label{ursell_clustering}
\sP(X_{z_1}=X_{z_2}=\dots=X_{z_{p+q}}=1)-\sP(X_{z_1}=\dots=X_{z_p}=1)\sP(X_{z_{p+1}}=\dots=X_{z_{p+q}}=1)= \sum\limits_{\pi'}\prod\limits_{B'}\mt^{(1,\dots,1)}(X_{z_i}:i\in B'),
\end{equation}
where the sum is over the partitions $\pi'$ of $\{1,2,\dots,p+q\}$, whose blocks $B'$ has at-least one term from each of the sets $\{1,2,\dots,p\}$ and $\{p+1,\dots,p+q\}$.} 

\noindent Therefore in the specified regimes $\left(\{\beta>0\} \backslash \{h=0; \beta\geq \beta_c(d)\}\right)$, the Ising model is an exponentially clustering spin model. Referring to Definition \ref{defn:clustspin}, we have $C_k=O(k^2)$, $c_k=1$ and the clustering function is the two point function $\phi_k(\cdot )=u_2(\cdot )$, where $u_2(t)=|\sP(X_{x_1}=X_{x_2}=1)-\sP(X_{x_1}=1)\sP(X_{x_2}=1)|$ for some $|x_1-x_2|=t$ and is exponentially decreasing in $t$.

\subsubsection{Determinantal point processes}
\label{sec:dpp}

Determinantal point processes have been of considerable interest in probability and statistical physics literature. We refer the reader to \cite[Chapter 4]{HKPV} for an introduction, \cite{lyons} for these processes on discrete structures and \cite{lyonssteif} for stationary determinantal point processes on $\mathbb{Z}^d$. \\

\noindent A stationary determinantal point process $\P$ on any discrete Cayley graph $G$ is defined to be a spin model whose probabilities are determined by the relations $\sP(\{x_1,\ldots,x_k\} \subset \P)=\det(K(x_i,x_j))_{i,j}$, where $K$ is a suitable non-negative definite real-valued kernel. From these relations it follows that the kernel $K$ has to be a contraction and  invariant under group action. Consider a stationary determinantal point process on the graph $G$ whose kernel $K$ decays exponentially with graph distance. 
%
It is known that (see \cite[Theorem 3.4]{BorceaBrandenLiggett2009}) determinantal point processes are negatively associated, which implies that they are quasi-associated (for related discussion see \cite[Definitions 1.1, 1.2 and 1.3]{BulinskiSuquet2001}). Hence, from the definition of quasi-association (see Section \ref{sec:quasi-asso-BL}) and by using arguments as in the proof of Proposition \ref{prop:blt_clust}, we have that the stationary determinantal point process is exponentially clustering with $C_k=O(k^2)$, $c_k=1$ and the clustering function $\phi_k(t) = |K(O,z)|$ for some $|z| = t$. 

\subsubsection{Other possible clustering spin models:}
\label{sec:ex_others}
The phenomenon of exponential decay of covariances (two point functions) of the fields in distance is sometimes termed as `massive', and it is a common feature in many statistical physics models. To mention a few, $\delta$-pinned models for Gaussian fields studied in \cite{Ioffe00,Velenik06,Bolthausen16}, stochastic interface models considered in \cite{Funaki05,Velenik06}, Gross-Neveu model considered in \cite{kopper} exhibit this behaviour in suitable regimes. These are only indicative references for the literature and by no means exhaustive. One can define spin models using these fields, for e.g. choosing a level set or determining the spins by value of the field in a finite neighbourhood. It is natural to expect that the corresponding spin models arising from these fields satisfy exponential clustering conditions. \\

\noindent We mentioned in Section \ref{sec:ising} that the Ising model satisfies clustering condition for high temperatures i.e, $\beta$ small enough. One generalisation of Ising model are {\em Gibbs spin models} by taking more general interaction potentials. Various Gibbs spin models with finite-range (or suitably local) potentials are also expected to satisfy clustering condition for high temperatures (see \cite{Holley1976,Gross1979,Kunsch1982}).




\subsection{\bf Applications to random cubical complexes.}
\label{sec:rcc}

We shall illustrate our central limit theorem for local and exponentially quasi-local statistics of clustering spin models using random cubical complexes. We shall now re-introduce cubical complexes a little more formally before detailing our results about the same. For more details on cubical complexes, see \cite{Kaczynski04,Werman16,Hiraoka16} or \cite[Section 6.4]{Klette2004}. \\

\noindent Set $Q = W_{1/2} = [-\frac{1}{2},\frac{1}{2}]^d,$ the unit cube centred at origin and $Q_x = Q + x$, the shifted cube for any $x \in \mR^d$. Given a spin model $\mu$, define $C(\mu) := \cup_{x \in \mu} Q_x.$ $C(\mu)$ can be viewed as a random subset of $\mR^d$ as is done in various applications in image analysis, stereology and mathematical morphology.  Alternatively, define $F_k(\mu) = \{ (x_0,\ldots,x_k) \in \mu^{(k)} : \cap_{i=0}^k Q_{x_i} \neq \emptyset \}$ as the $k$-faces of the cubical complex. By default, set $F_0(\mu) = \mu \subset \mZ^d$. Note that $F_k(\mu) = \emptyset$ for $k \geq 2d$. We shall denote elements of $F_k$ as $[x_0,\ldots,x_k]$. The {\em cubical complex} is defined as $\cK(\mu) := \cup_{k = 0}^{2d-1}F_k(\mu)$. The {\em $k$-skeleton} of $\cK$ is defined as $\cK^k(\mu) := \cup_{j = 0}^{k}F_k(\mu)$. Trivially, note that the $1$-skeleton is the graph with vertex set $\mu$ and edge-set $\{(x,y) \in \mu^2 :  Q_x \cap Q_y \neq \emptyset \}$. This is same as the Cayley graph on $\mZ^d$ defined in \eqref{eqn:cayleyzd}. We shall not give more details here on cubical complexes apart and will rather refer the reader to \cite{Kaczynski04}. Given two cubical complexes $\cK, \cL$, {\em a cubical homomorphism} is a map $f : \cK^0 \to \cL^0$ such that whenever $[x_0,\ldots,x_k] \in \cK$, then $[f(x_0),\ldots,f(x_k)] \in \cL$. We say that $f$ is a {\em cubical isomorphism} if it is a bijection and $f^{-1}$ is also a cubical homomorphism. 

\subsubsection{Local counts and intrinsic volumes:}
\label{para:subcomplex}
We first define an abstract class of local statistics that shall include various statistics of interest about cubical complexes and then state asymptotics for this local statistic using our main theorem (Theorem \ref{thm:main}). Let  $h :\{0,1\}^{W_k} \to \mR$ be a function of spin models on $W_k$. The local score function $\xi$ defined for $\mu \in \cN, x \in \mu$,
\begin{equation}
\label{eqn:localscoregen}
\xi(x,\mu) := h(\mu \cap W_k(x))
\end{equation}
Further, set the total mass $H^h(\mu) = H^{\xi}(\mu) := \sum_{x \in \mu} \xi(x,\mu).$ We shall now give some examples of $h$ and hence of $\xi$ as well. Let $z_1,\ldots,z_k \in \mZ^d$ be such that $\Gamma_k := \cK(\{z_1,\ldots,z_k\})$ is connected as a set (or equivalently the $1$-skeleton is connected).  Define the following two $h$ functions on $\mu \in \{0,1\}^{W_k}$ :
\begin{eqnarray}
\label{eqn:subcom} h_{\Gamma_k}(\mu)  & = & \frac{1}{k!} \sum_{(x_2,\ldots, x_k) \in \mu^{(k-1)}} \1[\cK(O,x_2\ldots, x_k) \cong \Gamma_k] \\
\label{eqn:compcom} \hat{h}_{\Gamma_k}(\mu) & = & \frac{1}{k!} \sum_{(x_2,\ldots,x_k) \in \mu^{(k-1)}} \1[\cK(O,x_2,\ldots,x_k) \cong \Gamma_k]  \1[\cup_{i=1}^k (\mu \cap W_1(x_i)) = \{O,x_2,\ldots,x_k\} ]
\end{eqnarray}
Observe that the the total mass $H^{h_{\Gamma_k}}(\mu)$ counts the number of isomorphic copies of $\Gamma_k$ in $\cK(\mu)$ and $H^{\hat{h}_{\Gamma_k}}(\mu)$ counts the number of components in $\cK(\mu)$ isomorphic to $\Gamma_k$.  These are two basic statistics of interest in combinatorial topology. \\

\noindent Let $\Delta_k := \{ \cup_{z \in \mu} Q_z : \mu \in \{0,1\}^{W_k} \}$ be the space of all cubical complexes in $W_k$, and $g : \Delta_k \to \mR$ be a real valued functional defined on cubical complexes. Then, setting 
$
h(\mu) :=  g(\cup_{z \in \mu} Q_z),
$
and by appropriately choosing $g$, we can get $\xi_j$ (for $0 \leq j \leq d$) to be a score function such that
\begin{equation}
\label{eqn:intrinsicvol}
H_j(\mu) := H^{\xi_j}(\mu) = V_j(C(\mu)),
\end{equation}
where $V_j$ is the $j$th intrinsic volume.  See \cite[Section 2.3.3]{Yogesh16} or \cite[Section 4]{Werman16} for the details of precise definition of $g$ and $\xi_j$'s. A precise definition of $\xi_j$ can be provided by using the finite additivity of intrinsic volumes. 
In fact, by the famed Hadwiger's characterisation theorem  (\cite[Theorem 14.4.6]{Schneider08}) the additive property on convex ring, motion invariance and appropriate continuity along with the assumption that $V_j(rA) = r^jV_j(A)$ uniquely characterises the $j$th intrinsic volume. $V_0$ is the famed {\em Euler-Poincar\'{e} characteristic}, and for a convex set $A$, $V_d(A)$ \& $V_{d-1}(A)$ are valuations of $A$ which equal the volume and surface measure of $A$, respectively. For more details on intrinsic 
volumes, we refer the reader to \cite{Schneider08}.

\noindent While first-order asymptotics of the above three statistics can be derived easily using ergodic theory arguments, results akin to CLTs are proven only in special cases. For example, if the spins are i.i.d., a central limit theorem for $H_j$'s (i.e., $V_j$'s) for $j = 0,\ldots, d$ is proven in \cite[Theorem 9]{Werman16}. A similar approach can be used to prove central limit theorems for $H^{\Gamma_k}$ and $H^{\hxi}$ in the case of i.i.d. spins. 
As an application of our main theorem, we now reduce the proof of such theorems for general spin models to that of deriving suitable variance lower bounds. 
\begin{theorem}
\label{thm:localstatappln}
Let $\P$ be a clustering spin model as in Definition \ref{defn:clustspin}, $h$ a general local function as defined in \eqref{eqn:localscoregen} and $\xi$ the score function induced by $h$. Set $H_n := H^{\xi}(\P_n)$. Then if for some $\nu >0$, it holds that
\[ \Var(H_n) = \Omega(w_n^{\nu}), \]
then we have that  $\Var(H_n)^{-1/2} (H_n - \EXP{H_n}) \stackrel{d}{\Rightarrow} N(0,1).$
\end{theorem}
%

We had earlier alluded to the use of integral geometric statistics in morpho-metric analysis of digital images and the simplest of such statistics being the intrinsic volumes. 
We shall now show an interesting application of Theorem \ref{thm:multimain} that shall give the joint distribution of the intrinsic volumes of random cubical complexes. 
\begin{corollary}
\label{cor_multivariateintrinsic}
Let $\P$ be a clustering spin model as in Definition \ref{defn:clustspin} and $\xi_i, i = 0,\ldots,d$ be score functions defined such that $H^{\xi_i}(\mu) = H_j(\mu)$ (see \eqref{eqn:intrinsicvol}) for any spin model $\mu$. Set $\bar{H}_n := (H_0(\P_n),\ldots,H_d(\P_n)).$ Then we have that
\[ \frac{\bar{H}_n - \EXP{\bar{H}_n}}{\sqrt{w_n}} \stackrel{d}{\Rightarrow} N(0,\Sigma) ,\]
where $\Sigma := (\Sigma(i,j))_{0 \leq i,j \leq d}$ is as defined in Theorem \ref{thm:multimain}.
\end{corollary}
In general, our multivariate central limit theorem can be applied to $(H_n^{h_1},\ldots,H_n^{h_k})$ for local functions $h_1,\ldots,h_k : \{0,1\}^{W_r} \to \mR$. For example, $h_i = h_{\Gamma_i}$ or $\hat{h}_{\Gamma_i}$ where $\Gamma_i, 1 = 1,\ldots,k$ are $k$ distinct connected complexes. Similar multivariate central limit theorems for subgraph counts have been proven for (Euclidean) Poisson point processes in \cite[Chapter 3]{Penrose03}. In the case of i.i.d. spins, one can derive multivariate CLTs for local statistics with rates can be derived using results in \cite{Krokowski2017}. For example, such a multivariate CLT for intrinsic volumes can be found in \cite[Section 4.2]{Krokowski2017}.

\subsubsection{Nearest neighbour graphs:}
\label{para:nng}

We shall now illustrate CLT for exponentially quasi-local score functions via the following graph, which often arises in computational geometry. It is called the {\em nearest-neighbour graph}, and is defined as follows: the vertices are points of $\mu$ and $(x,y)$ is a directed edge (called the {\em nearest neighbour edge (NNE)}) if $\mu \cap W_r(x) = \{x\}$ for all $r < |x-y|$. 
Then the nearest neighbour distance score function is given by 
\begin{equation}
\label{eqn:nnd}
\xi_{NN}(x,\mu) = \sum_{y \in \mZ^d} |x-y| \1[\mbox{$(x,y)$ an NNE}],
\end{equation}
where $|.|$ denotes the $L^1$ Euclidean distance.
The distribution of $\{ \xi_{NN}(x,\mu)\}_{x \in \mu}$ is important in understanding the structure of spacing of the points in $\mu$. 
In this case $H_{NN}(\mu)$ is the total edge length of  the `weighted' nearest neighbour graph on $\mu$ i.e., we count edges that occur as a NNE for both the points twice. It is possible to re-define the score function such that $H_{NN}(\mu)$ yields the total edge-length of the nearest neighbour graph but this shall complicate our analysis a little more. Hence, we avoid it here.  
\begin{theorem}
\label{thm:cltnnd}
Let $\P$ be an exponential clustering spin model as in \eqref{eqn:expclust}, $\xi := \xi_{NN}(\cdot )$ be score function of the nearest neighbour distance as defined in \eqref{eqn:nnd}. Set $H_n := H^{\xi}(\P_n)$. Then if for some $\nu >0$, it holds that
\[ \Var(H_n) = \Omega(w_n^{\nu}), \]
then we have that  $\Var(H_n)^{-1/2} (H_n - \EXP{H_n}) \stackrel{d}{\Rightarrow} N(0,1).$
\end{theorem}
\begin{proof}
By definition of  $\xi_{NN}$, we have that $R_{NN}(x,\mu) := R^{\xi}(x,\mu) = \inf \{ r : \mu \cap W_r(x) \neq \emptyset \}$. Thus, by Lemma \ref{lem:voidprob}, we obtain that for some $A,a',\nu > 0$,
\[ \sP(R_{NN}(x,\mu) > r) = \sP(\P_n \cap W_r(x) = \emptyset) \leq A e^{-a' r^{\nu}} .\]
This proves exponentially quasi-locality of the score function as in \eqref{eqn:quasilocal} with $c = \nu$. As for the  exponential growth condition \eqref{eqn:powergrowth}, note that
\[ |\xi_{NN}(x,\mu \cap W_t(x))| \leq |\partial W_t(x)| t \leq  Ct^d ,\]
i.e., we have the polynomial growth condition and thus by Remark \ref{rem:polygrowth}, $\xi_{NN}$ satisfies the $p$-moment condition \eqref{eqn:pmom} for all $p > 1$ as well.  Thus we have verified all the conditions in (2) of Theorem \ref{thm:main} and hence the central limit theorem follows.  \qed
\end{proof}

\noindent As mentioned in the introduction, one can also study $k$-nearest neighbour edges i.e., $(x,y)$ form a $k$-NNE if $y$ is among the first $k$ nearest neighbours of $x$ i.e., $\mu \cap W_r(x) \leq k$ for all $r < |x-y|$.  We remark here that in the case of nearest neighbour functional defined on $\mathbb{Z}^d$ with i.i.d. spins, one may use \cite[Theorem 2.1]{Penrose01} to estimate asymptotic variance. Further, the void probability bound in Lemma \ref{lem:voidprob} can be used to show exponentially quasi-locality of various score functions such as statistics (intrinsic volumes of faces, in-radii of cells) of Voronoi tessellations (\cite[Section 10.2]{Schneider08}) and other proximity graphs.

\subsubsection{Topology of random cubical complexes:}
\label{para:toprcc}
We shall now introduce a slightly more complicated score function from algebraic topology. With increased interest in applied topology, there have been many studies on topological invariants, like Betti numbers, of random complexes (see \cite{Bobrowski14,Kahle14}). 
The non-trivial Betti numbers of a ``nice" set $A \subset \mR^d$ (or a cubical complex $\cK$) denoted by $\beta_0(A),\ldots, \beta_{d-1}(A)$ are a measure of connectivity of the set $A$. For example, $\beta_0$ counts the number of connected components and $\beta_{d-1}(A)$ counts the number of connected components of $A^c$ minus one. The other Betti numbers are harder to describe and in simple terms, $\beta_k$ counts the number of `distinct' $(k+1)$-dimensional holes enclosed in the set $A$. A little more formally, one constructs a group called the `$k$th homology group'  with non-trivial $k$-dimensional cycles or equivalently $(k+1)$-dimensional holes and we define $\beta_k$ as the rank of the $k$th homology group. We assume that all the homology groups here are vector spaces (over an arbitrary field) and hence the ranks are well-defined.  For details, we refer the reader to \cite[Section 2.1]{Hiraoka16}, \cite[Section 6.4]{Klette2004}, or \cite{Edelsbrunner10,Kaczynski04}.
Topological properties of random cubical complexes and slightly more general models were investigated recently in \cite{Hiraoka16}. While first-order asymptotics (\cite[Section 2.4]{Hiraoka16})  were proven for general ergodic spin models, a central limit theorem (\cite[Theorem 2.11 and 2.12]{Hiraoka16}) again required the assumption of i.i.d. spins. We shall now state a central limit theorem for Betti numbers of random cubical complexes on subcritical spin models. Let $C(x,\mu)$ denote the component (i.e., maximal connected set or sub-complex) containing $x$ in $C(\mu)$ or $\cK(\mu)$. We shall use $C(x,\mu)$ to denote a subset of $\mZ^d$ as well as the corresponding union of unit cubes.  We shall say that $\P$ is a {\it subcritical spin model} (in terms of percolation) if for all $x,y \in \mZ^d$,
\[ \sP(y \in C(x,\P))  \leq C^* e^{-c^* |x-y|} ,\]
for some positive constants $C^*,c^*$. 
\begin{theorem}
\label{thm:cltnnd}
Let $\P$ be an exponential clustering spin model as in \eqref{eqn:expclust} and further, let $\P$ or $\P^c$ be a subcritical spin model. Then if, 
$ \Var(\beta_k(C(\P_n))) = \Omega(w_n^{\nu})$, for some $\nu >0$,
then $$\Var(\beta_k(C(\P_n)))^{-1/2} (\beta_k(C(\P_n)) - \EXP{\beta_k(C(\P_n))}) \stackrel{d}{\Rightarrow} N(0,1).$$ Further, we have that
\[  (w_n)^{-1/2}(\beta_0(C(\P_n)) - \EXP{\beta_0(C(\P_n))},\ldots,\beta_{d-1}(C(\P_n)) - \EXP{\beta_{d-1}(C(\P_n))}) \stackrel{d}{\Rightarrow} N(0,\Sigma), \]
where $\Sigma$ is as defined in Theorem \ref{thm:multimain} with $\xi_k$ as defined below in \eqref{eqn:xikbetti}.
\end{theorem}
%
%
\begin{proof}
We shall give the proof for the case when $\P$ is subcritical and then by duality arguments in algebraic topology extend it to the case when $\P^c$ is subcritical. We shall now define an appropriate score function and then prove exponentially quasi-locality of the same. For a fixed $k \in \{0,\ldots,d-1\}$, define the score function $\xi_k$ as
\begin{equation}
\label{eqn:xikbetti}
 \xi_k(x, \mu) := \frac{\beta_k(C(x,\mu))}{|C(x,\mu)|} ,
\end{equation} 
where $|C(x,\mu)|$ stands for the (vertex) cardinality of the connected component of $x$. Since all the cubes are of unit volume and disjoint, $|C(x,\mu)| = V_d(C(x,\mu)).$ By the property of Betti numbers of cubical complexes $\beta_k(C(x,\mu)) \leq |C(x,\mu)|$ (see \cite[Lemma 3.1]{Hiraoka16}), we have that $|\xi_k(x,\mu)| \leq 1$ and hence $\xi_k$ satisfies the power-growth condition \eqref{eqn:powergrowth} and $p$-moment condition \eqref{eqn:pmom} for all $p \geq 1$. Further, as $\beta_k(A \cup B) = \beta_k(A) + \beta_k(B)$ for disjoint $A$ and $B$, we have
\[ \beta_k(C(\mu)) =  \sum_{x \in \mu} \xi_k(x,\mu). \]
Let $d(x,\mu) := \inf \{s :  C(x,\mu) \subset W_{s-1}(x) \}$ be the $L_1$ distance from $x$ to the furthest vertex on the boundary of its connected component. Suppose $r > d(x,\mu)$ and $\mu'$ is a spin model such that $\mu' \cap W_r(x) = \mu \cap W_r(x)$. 
Trivially, we have that $C(x,\mu) = C(x,\mu')$ and hence $\xi_k(x, \mu) = \xi_k(x,\mu')$. In other words, configurations which differ
from $\mu$ outside $W_r(x)$ keep $C(x,\mu)$ unchanged.
Thus, we derive that $R(x, \mu) \leq d(x,\mu)$. Now, observe that
\[ \sP(d(x,\P_n) > t) \leq \sum_{y : | y - x| = t} \sP(y \in C(x,\P)) \leq C' t^{d-1} e^{-c^*t}, \]
and so $\xi_k(x,\P)$ is exponentially quasi-local as in \eqref{eqn:quasilocal}, implying $(\xi_k,\P)$ satisfies the assumptions in (2) of Theorem \ref{thm:main} and hence the CLT follows. \\

\noindent Now suppose that $\P^c$ is subcritical. By Remark \ref{rem:clust}, we have that $\P^c$ is also exponentially clustering as in \eqref{eqn:quasilocal}. Set $\mu^*_n := W_{n+1} \setminus W_n$. By the universal coefficient theorem for simplicial homology (\cite[Theorem 45.8]{Munkres84}) and Alexander duality (\cite[Theorem 16]{Spanier66})\footnote{Alexander duality shows isomorphism of $k$th homology group with $(d-k-1)$th cohomology group and the universal coefficient theorem gives equality of the rank of $(d-k-1)$th cohomology group with the corresponding homology group.}, we derive that  $\tilde{\beta}_k(C(\P_n)) = \tilde{\beta}_{d-k-1}(C((\P_n)^c)).$ Further from the  homotopy equivalence of $C((\mu \cap W_n) \cup W_n^c)$ and $C((\mu \cap W_n) \cup \mu^*_n)$ for any spin model $\mu$, we have that $\tilde{\beta}_{d-k-1}(C((\P_n)^c)) = \tilde{\beta}_{d-k-1}(C((\P^c)_n \cup \mu^*_n))$,  where $\tilde{\beta}_k(\cdot ) := \beta_k(\cdot ) - \1[k = 0]$, and we point out that $(\P_n)^c \neq (\P^c)_n$. 

Now, we can again define appropriate score functions $\xi'_k$ as above for $\tilde{\beta}_{d-k-1}(C((\P^c)_n \cup \mu_n))$ and by subcriticality of $\P^c$ show that $(\xi'_k,\P^c)$ satisfy the assumption (2) of Theorem \ref{thm:main}. Thus, the central limit theorem follows even in the case when $\P^c$ is subcritical.
The multivariate central limit theorem shall follow as we have shown that $(\xi_k,\P), (\xi_k',\P^c), k = 0,\ldots,d-1$ satisfy the assumptions in (2) of Theorem \ref{thm:multimain}.
\qed 
\end{proof}

\noindent As we noted earlier, the random cubical complexes considered in \cite{Hiraoka16} are more general and in particular, allow for lower-dimensional cubical subsets of $Q = [-1/2,1/2]^d$ to be present in the cubical complex without the entire cube $Q$ being present. To incorporate such a model into our set-up, one can consider a more general spin-model wherein the underlying lattice $\cL$ is generated by integer-valued translations of $\{ (a_1,\ldots,a_d) : a_i \in \{-1/2,0,1/2\} \}$ which are nothing but vertices of the cube $Q$ or mid-points of lower-dimensional cubes in $Q$. A spin being positive at a site in $\cL$ is equivalent to the particular lower-dimensional cube being present where the vertices are nothing but $0$-dimensional cubes. 

\section{Proofs}
\label{sec:Proofs}


\subsection{\bf Factorial Moment Expansion and clustering of random measures}
\label{sec:FME}

The key tool in our proofs is the clustering of $H^{\xi}_n$ derived using the factorial moment expansion introduced in \cite{Bartek95,Bartek97}. We shall now first recall the factorial moment expansion and then prove the clustering result. 
Recall that $V$ is the vertex set of the Cayley graph $G$, and
$\cN$ is the space of all spin configurations, or point measures, on $V$. Let $\prec$ be a total order on $V$ such that if  $|u| < |v|$ then $u \prec v$. One can construct such orders by starting with an arbitrary order on the generators $S$ and then constructing the corresponding lexicographic or co-lexicographic order on $V$.
For $\mu \in \cN$ on $V$ and $x \in V$, set $\mu_{|x}(\cdot ) := \mu(\cdot  \cap \{ y \in V : y \prec x\}).$ By local-finiteness of $\mu$ and the property of the ordering, $\mu_{|x}(V) \leq \mu(W_{|x|})$ and so $\mu_{|x}$ is a finite measure for all $x \in V$. We denote the null measure by $o$ i.e., $o(B) = 0$ for all subsets $B \subset V$. For a measurable function $\psi : \cN \to \mR$, non-negative integer $l$ and $x_1,\ldots,x_l \in V$, we define the factorial moment expansion (FME) kernels as follows : for $l = 0$, $D^0\psi(\mu) := \psi(o)$. For $l \geq 1$, 
\begin{equation}
\label{eqn:FME_kernel}
D^l_{x_1,\ldots,x_l}\psi(\mu) := \sum_{J \subset [l]}(-1)^{l-|J|}\psi(\mu_{|x_*} + \sum_{j \in J} \delta_{x_j}),
\end{equation}
where $[l] = \{1,\ldots,l\}$ and $x_* := \min \{x_1,\ldots,x_l\}$. Note that $D^l_{x_1,\ldots,x_l}\psi(\mu)$ is a symmetric function. We say that $\psi$ is $\prec$-continuous at $\infty$ if for all $\mu \in \cN$, we have that
\begin{equation}
\label{eqn:psi_cont}
\lim_{|x| \uparrow \infty} \psi(\mu_{|x}) = \psi(\mu).
\end{equation}
\begin{theorem}(\cite[Theorem 3.1]{Bartek97})
\label{FMEthm}
Let $\P$ be a simple point process on $V$ and $\psi : \cN \to \mR$ be $\prec$-continuous at $\infty$ and assume that 	for all $l \geq 1$
\be \label{eqn:FME_finite}
		\sumn_{(y_1,\ldots,y_l) \in V^l}\E[| D^{l}_{y_1,\ldots,y_l}\psi(\P \setminus \{y_1,\ldots,y_l\})|\1[\{y_1,\ldots,y_l\} \subset \P]]  < \infty
\ee

\be \label{eqn:FME_0}
\text{ and } \hspace{10 pt}		\frac{1}{l!} \sumn_{(y_1,\ldots,y_l) \in V^l} \E[D^{l}_{y_1,\ldots,y_l}\psi(\P \setminus \{y_1,\ldots,y_l\})\1[\{y_1,\ldots,y_l\} \subset \P]]  \to 0 \ {\rm{as}} \ l \to \infty,
\ee
where $\sum\limits_{}^{\neq}$ denotes that the summation is over only distinct $y_1,\ldots,y_l$. Then $\E[\psi(\P)]$ has the following {\em factorial moment expansion }
\be
\label{eqn:FME}
		\E[\psi(\P)] = \psi(o) +  \sum_{l=1}^{\infty}\frac{1}{l!} \sumn_{(y_1,\ldots,y_l) \in V^l} D^{l}_{y_1,\ldots,y_l}\psi(o)\rho^{(l)}(y_1,\ldots,y_l).
\ee
\end{theorem}
Though the actual statement of \cite[Theorem 3.1]{Bartek97} involves Palm expectations and probabilities, all of them vanish in this discrete case due to the simple form of conditional probabilities.  The proof of Theorem \ref{thm:clustering_rand_measure} shall use the above FME for $\mE(\psi(\P_n))$, where $\psi(\mu)$ is  the following product of the score functions
	\begin{equation}
	\label{eqn:psi_xi}
	\psi(\mu)= \psi_{k_1,\ldots,k_p}(x_1,\ldots,x_p;\mu):=\prod_{i=1}^p\xi(x_i,\mu)^{k_i},
	\end{equation}	
with $k_1,\ldots,k_p\ge 1$. For ease of stating the FME, we shall consider the following modified functional $\psip$
\begin{equation}
	\label{eqn:psip}
	\psip(\mu)= \psip_{k_1,\ldots,k_p}(x_1,\ldots,x_p;\mu) := \psi(\mu + \sum_{j=1}^p\delta_{x_j}) = \prod_{i=1}^p\xi\Bigl(x_i,\mu+\textstyle{\sum_{j=1}^p}\delta_{x_j}\Bigr)^{k_i}.
	\end{equation}
\begin{proposition}
\label{prop:FME_psi}
Let $\P$ be a clustering spin model on $V$ as in Definition \ref{defn:clustspin} and $\xi$ be a local statistic (i.e., $R(z,\P) \leq r$ for all $z$ a.s.) as in \eqref{eqn:local}. Then for distinct $x_1,\ldots,x_p\in V$, non-negative integers $k_1,\ldots,k_p$ and $1\le n\le\infty$,
		the functional $\psip$ at \eqref{eqn:psip} admits the FME
		\begin{eqnarray} 
		\nonumber & m^{k_1,\ldots,k_p}_n(x_1,\ldots,x_{p}) = \sE(\psi_{k_1,\ldots,k_p}(x_1,\ldots,x_p;\P_n)) \\
		\label{eqn:FMEpsip} & = \psip_{k_1,\ldots,k_p}(x_1,\ldots,x_p;o) + \sum_{l=1}^{L_p}\frac{1}{l!} \!\!\!\!\!\!\!\!\!
		\sumn\limits_{\substack{(y_1,\ldots,y_l) \in \\ (\cup_{i=1}^pW_{r,n}(x_i))^{(l)}}} \!\!\!\!\!\!\!\!\!\!\!\!
		D^{l}_{y_1,\ldots,y_l}\psip_{k_1,\ldots,k_p}(x_1,\ldots,x_p;o) \rho^{(p+l)}(x_1,\ldots,x_p,y_1,\ldots,y_l), 
		\end{eqnarray}
where $L_p = |\cup_{i=1}^pW_r(x_i)|$ and $W_{r,n}(x) := W_r(x) \cap W_n$.
\end{proposition}

\begin{proof}
Firstly, by the local property of $\xi$, we have that if $y_1,\ldots,y_l \in V$ such that $y_k \notin \cup_{i=1}^pW_r(x_i)$ for some $k \leq l$, then for any spin model $\mu$ on $V$,
\begin{equation}
\label{eqn:kernel_trivial}
D^l_{y_1,\ldots,y_l}\psi_{k_1,\ldots,k_p}(x_1,\ldots,x_p;\mu) = 0.
\end{equation}
The proof follows easily by noting that in \eqref{eqn:FME_kernel}, the terms of the form $J \cup \{k\}$ and $J$ with $J \subset [l] \setminus \{k\}$ cancel out each other due to the locality property of $\xi$ (see \cite[(3.11)]{Yogesh16} for a detailed proof).\\ \\
\noindent Thus, from \eqref{eqn:kernel_trivial} we have that the FME kernels  for $ \psip_{k_1,\ldots,k_p}(x_1,\ldots,x_p;\mu)$ (i.e., $D_{y_1,\ldots,y_l}\psip_{k_1,\ldots,k_p}(x_1,\ldots,x_p;\mu)$ ) are non-trivial only if $y_1,\ldots, y_l \in \cup_{i=1}^pW_r(x_i).$  Since the correlation functions are zero if the coordinates repeat, we have that $\rho^{(l)}(y_1,\ldots,y_l)$ is non-trivial for $y_1,\ldots, y_l \in \cup_{i=1}^pW_r(x_i)$ only if $l \leq L_p$ where $L_p$ is as defined in the proposition. Thus, for $l > L_p$
\[ \sumn_{(y_1,\ldots,y_l) \in V^l} \E[D^{l}_{y_1,\ldots,y_l}\psi_{k_1,\ldots,k_p}(x_1,\ldots,x_p;\P_n)\1[\{y_1,\ldots,y_l\} \subset \P_n \cap (\cup_{i=1}^pW_r(x_i)) ]] = 0.\]
Since local property of $\xi$ guarantees $\prec$-continuity of $\psi$, we only need to prove \eqref{eqn:FME_finite} to verify the conditions of Theorem \ref{FMEthm} and then we shall show that it reduces to \eqref{eqn:FMEpsip}. Setting $K_p = \sum_{i=1}^p k_i$, note that by definition of the FME kernels in \eqref{eqn:FME_kernel} and from the definition of $\|\xi\|_{\infty}$ at the beginning of the proof we trivially have that
\[ | D^{l}_{y_1,\ldots,y_l}\psi_{k_1,\ldots,k_p}(x_1,\ldots,x_p;\P_n)| \leq 2^l \|\xi\|_{\infty}^{K_p},\]
where $\|\xi\|_{\infty} = \|\xi_r\|_{\infty}$ is as defined in Remark \ref{rem:pmomlocal}. Now using triviality of FME kernels as in \eqref{eqn:kernel_trivial}, we get that for any $l \geq 1$, 
\[ \sumn_{(y_1,\ldots,y_l) \in V^l} \E[|D^{l}_{y_1,\ldots,y_l}\psi_{k_1,\ldots,k_p}(x_1,\ldots,x_p;\P_n)|] \leq 2^l \|\xi\|_{\infty}^{K_p} L_p^l < \infty, \]
and thus \eqref{eqn:FME_finite} is valid for $\psi_{k_1,\ldots,k_p}(x_1,\ldots,x_p)$ i.e.,
\[ m^{k_1,\ldots,k_p}_n(x_1,\ldots,x_p) =  \psi_{k_1,\ldots,k_p}(x_1,\ldots,x_p;o) + \sum_{l=1}^{L_p}\frac{1}{l!}\sumn_{\substack{(y_1,\ldots,y_l)\\ \in (\cup_{i=1}^pW_{r,n}(x_i))^l}} D^{l}_{y_1,\ldots,y_l}\psi_{k_1,\ldots,k_p}(x_1,\ldots,x_p;o) \rho^{(l)}(y_1,\ldots,y_l).\]                                                                                                                                                                                                                                                                                                                                                                                                                                                                                                                                                                                                                                                                                                                                                                                                                                                                                                                                                                                                                                                                                                                                               
But by definition of $\xi$, we have that if $x_i \notin \mu$ for some $1 \leq i \leq p$, then $\psi_{k_1,\ldots,k_p}(x_1,\ldots,x_p;\mu) = 0$.  Thus, 
$D^l_{y_1,\ldots,y_l}\psi_{k_1,\ldots,k_p}(x_1,\ldots, x_p;o) = 0$ unless $ \{x_1, \ldots, x_p\} \subset \{y_1, \ldots, y_l \}$. Rearranging and relabelling the coordinates, we shall assume that the summand in the above FME expansion is non-trivial only for $\{x_1,\ldots,x_p,y_1,\ldots,y_l\} \subset V$ for some $l \geq 0$. Thus, we get that
\[ m^{k_1,\ldots,k_p}_n(x_1,\ldots,x_p) =  \sum_{l=0}^{L_p}\frac{1}{(p+l)!}\sumn_{\substack{(z_1,\ldots,z_{p+l})\\ \in (\cup_{i=1}^pW_{r,n}(x_i))^{p+l}}} D^{p+l}_{z_1,\ldots,z_{p+l}}\psi(o) \rho^{(p+l)}(z_1,\ldots,z_{p+l}) \1 [ \{x_1, \ldots, x_p\} \subset  \{z_1,\ldots,z_{p+l} \} ]. \]                                                                                                                                                                                                                                                                                                                                                                                                                                                                                                                                                                                                                                                                                                                                                                                                                                                                                                                                                                                                                                                                                                                                               
Now fixing an $l \geq 0$. Again combining the definition of $\xi$ with the definition of FME kernels as in \eqref{eqn:FME_kernel}, we have that 
\[ D^{p+l}_{x_1,\ldots,x_p,y_1,\ldots,y_l}\psi(o) =  \sum_{J \subset [l]} (-1)^{l-|J|} \psi(\sum_{i=1}^p\delta_{x_i} + \sum_{j \in J}\delta_{y_j})  = \sum_{J \subset [l]} (-1)^{l-|J|} \psip(\sum_{j \in J}\delta_{y_j}) = D^l_{y_1,\ldots,y_l}\psip(o) .\]
By the invariance of the RHS under permutations of $\{x_1,\ldots,x_p,y_1,\ldots,y_l\}$, we have that 
\[ D^{p+l}_{z_1,\ldots,z_{p+l}}\psi(o) \rho^{(p+l)}(z_1,\ldots,z_{p+l}) \1 [ \{x_1, \ldots, x_p\} \subset  \{z_1,\ldots,z_{p+l} \} ] = D^l_{y_1,\ldots,y_l}\psip(o), \]
where we have assumed that $\{z_1,\ldots,z_{p+l} \} = \{x_1,\ldots,x_p,y_1,\ldots,y_l\}.$ Now the proof of \eqref{eqn:FMEpsip} follows by correctly counting $\{z_1,\ldots,z_{p+l} \} = \{x_1,\ldots,x_p,y_1,\ldots,y_l\}.$ \qed
\end{proof} 


\begin{proof}(Theorem \ref{thm:clustering_rand_measure})
{\em Case (1) - `Local' score functions :}  Firstly, we shall prove the bounded stabilisation radius case i.e., assume a.s. for all $x \in V$, $R(x,\P) \leq r.$  Choose $s > 2r$. Fix $x_1,\ldots,x_{p+q}$ and assume that $s = d(\{x_1,\ldots,x_p\},\{x_{p+1},\ldots,x_{p+q}\}).$  Setting 
\begin{equation}
\label{e:lplq}
L_p = |\cup_{i=1}^p W_r(x_i)|, \, \, \, L_q = |\cup_{i=1}^q W_r(x_{p+i})|, \, \, \, L_{p+q} = |\cup_{i=1}^{p+q} W_r(x_i)|,
\end{equation}
we see that trivially $L_{p+q} = L_p + L_q$ as the sets $\cup_{i=1}^p W_r(x_i), \cup_{i=1}^q W_r(x_{p+i})$ are disjoint. Note that $L_p \leq pw_r$ and similarly for $L_q, L_{p+q}.$ Further set $K_p = \sum_{i=1}^pk_i, K_q = \sum_{i=1}^qk_{p+i}, K = \sum_{i=1}^{p+q}k_i$ and let $\psip$ be as defined in \eqref{eqn:psip}. Thus applying Proposition \ref{prop:FME_psi} we obtain that 
\begin{eqnarray} \label{eqn:fmep+q}
 && m^{k_1,\ldots,k_{p+q}}(x_1,\ldots ,x_{p+q}) \no\\
&=&  \sum_{l=0}^{L_{p+q}}\frac{1}{l!}\sum\limits_{\substack{(y_1,\ldots  y_l) \in (\cup_{i=1}^{p+q}W_{r,n}(x_i))^{(l)}}} D^{l}_{y_1,\ldots ,y_l}\psip_{k_1,\ldots ,k_{p+q}}(x_1,\ldots ,x_{p+q};o) \rho^{(l+p+q)}(x_1,\ldots ,x_{p+q},y_1,\ldots ,y_l), \no\\
& = & \sum_{l=0}^{L_{p+q}}\frac{1}{l!}\sum_{j=0}^l\frac{l!}{j!(l-j)!}
\sum^*_y D^{l}_{y_1,\ldots ,y_l}\psip_{k_1,\ldots ,k_{p+q}}(x_1,\ldots ,x_{p+q};o) \no  \times \rho^{(l+p+q)}(x_1,\ldots ,x_{p+q},y_1,\ldots ,y_l), \no \\
 & = & \sum_{l=0}^{L_{p+q}}\sum_{j=0}^l \frac1{j!(l-j)!}
\sum^*_y
\sum_{J \subset [l]} (-1)^{l-|J|} \psip_{k_1,\ldots ,k_{p+q}}(x_1,\ldots ,x_{p+q}; \sum_{j \in J} \delta_{y_j}) \, \rho^{(l+p+q)}(x_1,\ldots ,x_{p+q},y_1,\ldots ,y_l),
\end{eqnarray} 
where, $\sum\limits_{y}^{*}$ denotes summation over the set $\{(y_1,\ldots  y_l) \in (\cup_{i=1}^pW_{r,n}(x_i))^{(j)} \times (\cup_{i=1}^qW_{r,n}(x_{p+i}))^{(l-j)}\}$. Proceeding again as above, and writing $\sum\limits_{y,z}^{*}$ as summation over the set
$\{(y_1,\ldots ,y_{l_1},z_1,\ldots , z_{l_2}) \in \\ (\cup_{i=1}^pW_{r,n}(x_i))^{(l_1)} \times (\cup_{i=1}^qW_{r,n}(x_{p+i}))^{(l_2)}\}$,
we derive that

%
\begin{eqnarray}
&  & m^{k_1,\ldots,k_p}(x_1,\ldots ,x_p) m^{k_{p+1},\ldots ,k_{p+q}}(x_{p+1},\ldots ,x_{p+q})\no \\
& = & \sum_{l_1=0}^{L_p}\sum_{l_2=0}^{L_q} \frac1{l_1!\,l_2!}
\sum^*_{y,z}
\sum_{J_1 \subset [l_1], J_2 \subset [l_2]} (-1)^{l_1 + l_2 - |J_1| - |J_2|} \psip_{k_1,\ldots ,k_p}(x_1,\ldots ,x_p; \sum_{j \in J_1} \delta_{y_j}) \no\\
&& \times \psip_{k_{p+1},\ldots ,k_p}(x_{p+1},\ldots ,x_{p+q}; \sum_{j \in J_2} \delta_{z_j})
\,\, \rho^{(l_1+p)}(x_1,\ldots ,x_{p+q},y_1,\ldots ,y_{l_1}) \rho^{(l_2+q)}(x_{p+1},\ldots ,x_{p+q},y_1,\ldots ,y_{l_2}), 
\no \\
& = & \sum_{l=0}^{L_p+L_q}\sum_{j=0}^{l}\frac{1}{j! (l-j)!} 
\sum^*_y \sum_{J_1 \subset [j], J_2 \subset [l] \setminus [j]} (-1)^{l - |J_1| - |J_2|} \psip_{k_1,\ldots ,k_p}(x_1,\ldots ,x_p; \sum_{i \in J_1} \delta_{y_i}) \no \\
&& \times \psip_{k_{p+1},\ldots ,k_{p+q}}(x_{p+1},\ldots ,x_{p+q}; \sum_{i \in J_2} \delta_{y_i}) 
\,\, \rho^{(j+p)}(x_1,\ldots ,x_{p+q},y_1,\ldots ,y_j) \rho^{(l-j+q)}(x_{p+1},\ldots ,x_{p+q},y_{j+1},\ldots ,y_l), \no \\
& = & \sum_{l=0}^{L_{p+q}}\sum_{j=0}^{l}\frac{1}{j! (l-j)!}
\sum^*_y \sum_{J \subset [l]} (-1)^{l - |J|} \psip_{k_1,\ldots ,k_{p+q}}(x_1,\ldots ,x_{p+q}; \sum_{i \in J} \delta_{y_i}) \no \\
\label{eqn:fmepq} 
&  & \times \rho^{(j+p)}(x_1,\ldots ,x_{p+q},y_1,\ldots ,y_j) \rho^{(l-j+q)}(x_{p+1},\ldots ,x_{p+q},y_{j+1},\ldots ,y_l),
\end{eqnarray}
where in the last equality we have used the fact that for all $J \subset [l]$ with $J_1 = J \cap [j], J_2 = J \setminus J_1$, we have that
\[ \psip_{k_1,\ldots ,k_{p+q}}(x_1,\ldots ,x_{p+q}; \sum_{i \in J} \delta_{y_i}) = \psip_{k_1,\ldots ,k_p}(x_1,\ldots ,x_p; \sum_{i \in J_1} \delta_{y_i})\psip_{k_{p+1},\ldots ,k_p}(x_{p+1},\ldots ,x_{p+q}; \sum_{i \in J_2} \delta_{y_i}) .\]
This follows from the definitions of $\psip$ (see \eqref{eqn:psip}), $R(\cdot ,\P)$ and that $R(x_i,\P) \leq r$ for all $i \in [p+q]$. We now set $\tilde{L}_{p+q} = L_{p+q} + p + q$ and note that $\tilde{L}_{p+q} + p + q \leq (p+q)(1 + w_r) \leq K(1+w_r)$. Comparing \eqref{eqn:fmep+q} with \eqref{eqn:fmepq}, and using Definition \ref{defn:clustspin}
and Remark \ref{rem:pmomlocal}, we have that
 \begin{eqnarray} \label{eqn:bddmm_final}
 &  & |  m^{k_1,\ldots ,k_{p+q}}(x_1,\ldots ,x_{p+q}) -  m^{k_1,\ldots,k_p}(x_1,\ldots ,x_p) m^{k_{p+1},\ldots ,k_{p+q}}(x_{p+1},\ldots ,x_{p+q})|  \no \\
 & \leq &   \sum_{l=0}^{L_{p+q}}\sum_{j=0}^{l}\frac{C_{l+p+q}\phi(c_{l+p+q}s)}{j! (l-j)!}  \!\!\!\!\!\!\!\!\!\!\!\!
 \sum\limits_{\substack{(y_1,\ldots ,y_l) \in \\ (\cup_{i=1}^pW_{r,n}(x_i))^{(j)} \times (\cup_{i=1}^qW_{r,n}(x_{p+i}))^{(l-j)}}} \!\!\!\!\!\!\!\!\!\!\!\!
 \sum_{J \subset [l]} |\psip_{k_1,\ldots ,k_{p+q}}(x_1,\ldots ,x_{p+q}; \sum_{i \in J} \delta_{y_i})| , \, \, \, \, \, \,  \no \\ 
\label{eqn:bddmm_final} & \leq & \phi(c_{\tilde{L}_{p+q}}s) \sum_{l=0}^{L_{p+q}}\sum_{j=0}^{l}\frac{C_{l+p+q}}{j! (l-j)!}  2^l \|\xi\|_{\infty}^{K} L_{p+q}^l \,\,\,\,\,\,
\leq \,\,\,\,\,\, \phi(c_{K(1+w_r)}s) \|\xi\|_{\infty}^{K} \sum_{l=0}^{Kw_r} \frac{C_{l+K}(4 Kw_r)^l}{l!}.
 \end{eqnarray}
 Thus, we have proven clustering of mixed moment functions for local score functions with $\tilde{\phi}(\cdot ) = \phi(\cdot )$, $\tilde{C}_{K} = 
 \|\xi\|_{\infty}^{K} \sum_{l=0}^{Kw_r} \frac{C_{l+K}(4 Kw_r)^l}{l!}, \tilde{c}_K = c_{K(1+w_r)}.$ Also trivially note that summability of $\phi$ as in \eqref{eqn:phisum} implies summability of $\tilde{\phi}$ as well. \\
 
\noindent {\em Case (2) - `Quasi-local' score functions :}  Let us fix $x_1,\ldots,x_{p+q}$, and assume that $s = d(\{x_i\}_{i=1}^p,\{x_{p+j}\}_{j=1}^q).$ Further without loss of generality, let $s > 4$ and $n \gg s$.  Set $t = t(s) :=  (s/4)^{\gamma}$ for a $\gamma \in (0,1)$ to be chosen later. Define for all $x_i$,
 \[ \txi(x_i, \P_n)  = \xi(x,\P_n) \1 [R(x_i,\P_n) \leq t], \]
By definition of radius of stabilisation, we have that $\txi(z,\P) = \xi(z,\P \cap W_t(z))$ for all $z \in V$ and so $\txi$ is a local statistic with radius of stabilisation $\tilde{R}(x,\P_n) \leq t$. Further set $\tilde{m}^{k_1,\ldots ,k_p}(x_1,\ldots ,x_p;n) := \sE[\prod_{i=1}^p\txi(x_i,\P_n)^{k_i}]$. Now, by H\"{o}lder's inequality, the moment condition \eqref{eqn:pmom} and exponentially quasi-locality \eqref{eqn:quasilocal}, we have that
\begin{eqnarray}
& & | m^{k_1,\ldots ,k_p}(x_1,\ldots ,x_p;n) - \tilde{m}^{k_1,\ldots ,k_p}(x_1,\ldots ,x_p;n)|  \leq | \sE[\prod_{i=1}^p\xi(x_i,\P_n)^{k_i}] - \sE[\prod_{i=1}^p\txi(x_i,\P_n)^{k_i}]| \no \\
& \leq & M_p^{K_p/(K_p+1)} (pA\varphi(t))^{1/(K_p+1)} \, \, \leq \, \, M_{K}^{K/(K+1)}(AK\varphi(t))^{1/(K+1)}. \no
\end{eqnarray}
Further, for any reals $A',B',A'',B''$ with $|B''| \leq |B'|$ we have that $|A''B'' - A'B'| \leq (|A'| + |B'|)(|A''- A'| + |B'' - B'|)$. Using this bound and H\"{o}lder's inequality as above, we have that
\begin{eqnarray}
&  & | m^{k_1,\ldots,k_{p+q}}(x_1,\ldots, x_{p+q};n) -  m^{k_1,\ldots,k_p}(x_1,\ldots ,x_p;n) m^{k_{p+1},\ldots,k_{p+q}}(x_{p+1},\ldots, x_{p+q};n) | \no \\
&  & \no \\
& \leq & | m^{k_1,\ldots ,k_{p+q}}(x_1,\ldots ,x_{p+q};n) - \tilde{m}^{k_1,\ldots ,k_{p+q}}(x_1,\ldots ,x_{p+q};n)|  \no \\
&  & +  | m^{k_1,\ldots,k_p}(x_1,\ldots ,x_p;n) m^{k_{p+1},\ldots,k_{p+q}}(x_{p+1},\ldots, x_{p+q};n) - \tilde{m}^{k_1,\ldots,k_p}(x_1,\ldots ,x_p) \tilde{m}^{k_{p+1},\ldots,k_{p+q}}(x_{p+1},\ldots, x_{p+q}) | \no \\
&  & + | \tilde{m}^{k_1,\ldots,k_{p+q}}(x_1,\ldots, x_{p+q};n) -  \tilde{m}^{k_1,\ldots,k_p}(x_1,\ldots ,x_p;n) \tilde{m}^{k_{p+1},\ldots,k_{p+q}}(x_{p+1},\ldots, x_{p+q};n) |, \no \\
&  & \no \\
& \leq & 5M_K^{2K/(K +1)} (AK\varphi(t))^{1/(K+1)}  + | \tilde{m}^{k_1,\ldots,k_{p+q}}(x_1,\ldots, x_{p+q}) -  \tilde{m}^{k_1,\ldots,k_p}(x_1,\ldots ,x_p;n) \tilde{m}^{k_{p+1},\ldots,k_{p+q}}(x_{p+1},\ldots, x_{p+q};n) |, \no \\
\label{eqn:bd_mm_truncation}
\end{eqnarray}
Set $L_p(t) =  |\cup_{i=1}^pW_t(x_i)|$ and similarly $L_q(t), L_{p+q}(t).$ Now using \eqref{eqn:bddmm_final} for $\txi$, we have that 
\[ | \tilde{m}^{k_1,\ldots,k_{p+q}}(x_1,\ldots, x_{p+q}) -  \tilde{m}^{k_1,\ldots,k_p}(x_1,\ldots ,x_p) \tilde{m}^{k_{p+1},\ldots,k_{p+q}}(x_{p+1},\ldots, x_{p+q}) |  \leq \phi(c_{\tilde{L}_{p+q}}s) \|\xi_t\|_{\infty}^{K} \sum_{l=0}^{Kw_t} \frac{C_{l+p+q}(4 Kw_t)^l}{l!} .\]
 %
Now using the polynomial growth of $w_t$ (see Definition \eqref{def:amen}) exponential growth condition \eqref{eqn:powergrowth} and  exponential clustering condition \eqref{eqn:expclust}, we derive that
\begin{eqnarray}
&  & | \tilde{m}^{k_1,\ldots,k_{p+q}}(x_1,\ldots, x_{p+q}) -  \tilde{m}^{k_1,\ldots,k_p}(x_1,\ldots ,x_p) \tilde{m}^{k_{p+1},\ldots,k_{p+q}}(x_{p+1},\ldots, x_{p+q}) | \no  \\
& \leq & c'_1\exp(-cs^b) \times \exp(CKt^{\kappa}) \times c'_3 \exp(c'_4t^{d/a}) \, \, \leq  \, \, c'_5\exp(-cs^b) \times \exp(c'_6s^{\gamma (\kappa+ d/a)}), \no
\end{eqnarray}
where all the $c'_i$ are constants depending on $a,d$ and $K$. Thus, if we choose $\gamma$ such that $\gamma (\kappa + d/a) < b/2$, we have that
 \[ | \tilde{m}^{k_1,\ldots,k_{p+q}}(x_1,\ldots, x_{p+q}) -  \tilde{m}^{k_1,\ldots,k_p}(x_1,\ldots ,x_p) \tilde{m}^{k_{p+1},\ldots,k_{p+q}}(x_{p+1},\ldots, x_{p+q}) |  \leq \tilde{C}_K e^{-\tilde{c}_Ks^{b/2}}, \]
for two constants $C'_k,c'_k \in (0,\infty)$ as required to complete the proof of the theorem.  \qed
 \end{proof}

\subsection{\bf Proof of the general central limit theorem - Theorem \ref{thm:cltrandfields}}
\label{sec:proofcltrandfields}

As stated earlier, we shall only be giving a sketch of the proof here as a similar theorem or variant has appeared in all the papers using the cumulant method such as\cite{Malyshev75,Martin80,Janson1988,Baryshnikov2005,Nazarov12,Yogesh16,Feray2016a,Bjorklund2017}. For more details, see the proof in \cite[Section 4.4]{Yogesh16}. We shall use the arguments therein adapted to the discrete Cayley graph case by essentially setting $f \equiv 1$ and defining the moment measures appropriately.\\ \\
\noindent We shall first define cumulants and state their relations to moments. For a random variable $Y$ with all moments being finite, the cumulants $S_k(Y), k \geq 1$ are formally defined as the coefficients in the power series expansion of the log Laplace transform of $Y$ in the negative domain i.e.,

\[ \log \sE(e^{tY}) = \log(1 + \sum_{k \geq 1}M_k(Y)t^k) = \sum_{k \geq 1}S_k(Y)t^k, \]
where $M_k(Y) := \sE(Y^k)$ is the $k$th moment of $Y$. Further, by formally manipulating the above series expansion, we can derive the following useful relation between moments and cumulants (see \cite[Proposition 3.2.1]{Peccati2011}) :
\begin{equation}
\label{eqn:momcum}
S_k = \sum_{\gamma = \{\gamma(1),\ldots,\gamma(l)\} \in \Pi[k]} (-1)^{l-1}(l-1)!\prod_{i=1}^lM_{|\gamma(l)|},
\end{equation}
where $\Pi[k]$ denotes the set of all unordered partitions of the set $[k] =\{1,\ldots,k\}$, a partition $\gamma \in \Pi[k]$ is represented as $\gamma = \{\gamma(1),\ldots,\gamma(l)\}$ with $l$ representing the number of classes and $\gamma(i)$, the elements in the $i$th class. $|\gamma(i)|$ denotes the cardinality of $\gamma(i)$ and we shall also denote $l$ by $|\gamma|$. Equipped with this relation, cumulants can be defined as long as all moments exist without being concerned about the existence of the log Laplace transform. We refer the reader to \cite[Chapter 3]{Peccati2011} for more details on cumulants. 

\begin{proof}[Proof of Theorem \ref{thm:cltrandfields}]
For convenience, we shall drop the superscript $X$ in our notations referring to the underlying random field $\X_n$. Let $M_{k,n},S_{k,n}$ denote the moments and cumulants of $H_n$ respectively and the moments and cumulants of $\bar{H}_n = \frac{H_n - \sE(H_n)}{\sqrt{\Var(H_n)}}$ are denoted by $\bar{M}_{k,n}, \bar{S}_{k,n}$ respectively. The existence of the above moments and cumulants follow from the assumption of moment condition in the theorem and H\"{o}lder's inequality. By \eqref{eqn:momcum}, we have that $\bar{S}_{1,n} = 0, \bar{S}_{2,n} = 1$ for all $n \geq 1$. Further, for a random variable $Y$ and a constant $c \in \mR$, we have that $S_k(cY) = c^kS_k(Y)$ for all $k \geq 1$ and $S_k(Y+c) = S_k(Y)$ for all $k \geq 2$. Thus we obtain that for all $k \geq 2$, 
\[ \bar{S}_{k,n} = \Var(H_n)^{-k/2}S_{k,n}.\]
Now by the above relation, the variance lower bound assumption and an extension of the classical cumulant method using Marcinkiewicz's theorem (see \cite[Theorem 1]{Janson1988}), we have the required normal convergence if we show that for all $k \geq 3$
\begin{equation}
\label{eqn:highcumto0}
S_{k,n} = O(w_n). 
\end{equation}
The rest of the proof will consist of showing the above bound. First define {\em mixed moment functions} of the random field $\X_n$ as $m^{(k_1,\ldots,k_p)}(x_1,\ldots,x_p;n) = \sE(\prod_{i=1}^nX_{n,x_i}^{k_i})$ for $1 \leq n < \infty$, $x_1,\ldots,x_p \in W_n$ and $k_1,\ldots,k_p \geq 1$.  Now define the {\em truncated mixed moment functions} or {\em Ursell functions}  $\mt^{(k_1,\ldots,k_p)}(x_1,\ldots,x_p;n)$ as follows : Set $\mt^{(1)}(\cdot ) = m^{(1)}(\cdot )$ for $1 \leq i \leq p$ and inductively,
\begin{equation}\label{eqn:urselljointmoment}
 \mt^{(k_1,\ldots,k_p)}(x_1,\ldots,x_p;n) := m^{(k_1,\ldots,k_p)}(x_1,\ldots,x_p;n) - \sum_{\gamma = \{\gamma(1),\ldots,\gamma(l)\} \in \Pi[p] \atop l > 1}\prod_{i=1}^{l}\mt^{(k_j: j \in \gamma(i))}(x_j ; j \in \gamma(i);n). 
\end{equation}

\noindent The mixed moment functions and Ursell functions exist because of the moment condition assumed in the theorem and H\"{o}lder's inequality. The Ursell functions are crucial to our analysis because by using the above definition and that of cumulant, one can derive that (see  \cite[Section 2]{Hegerfeldt1985} or \cite[Section 4.4]{Yogesh16})
\[ S_{k,n} = \sum_{k_1+\cdots+k_p = k \atop k_1,\cdots,k_p \geq 1} \sum_{x_1,\ldots,x_p \in W_n}\mt^{(k_1,\ldots,k_p)}(x_1,\ldots,x_p;n).
\]
Observe that for all $k \geq 1$, $\Pi[k]$ is finite and so the first summand is over finitely many terms for all $n \geq 1$. Hence the proof of \eqref{eqn:highcumto0} follows if for any $k \geq 1$ and all $k_1,\ldots,k_p \geq 1$ such that $k_1+\cdots+k_p = k$, we show that 
\begin{equation}
\label{eqn:ursellbound}
\sum_{x_1,\ldots,x_p \in W_n}|\mt^{(k_1,\ldots,k_p)}(x_1,\ldots,x_p;n)| = O(w_n).
\end{equation}
Now onwards, we fix a $1 \leq n < \infty$ and drop the reference to $n$ in the notation for Ursell functions and mixed moment functions. A partition $\gamma = \{\gamma(1),\ldots,\gamma(l)\}$ is said to {\em refine} a partition $\sigma = \{\sigma(1),\ldots,\sigma(l_1)\}$ if for all $i \in [l]$, $\gamma(i) \subset \sigma(j)$ for some $j \in [l_1]$. Else, we say that $\gamma$ {\em mixes} $\sigma$. Using this definition and the definition of Ursell functions, we derive that for any $I \subsetneq [p]$
\[ m^{(k_j: j \in I)}(x_j ; j \in I)m^{(k_j: j \in I^c)}(x_j ; j \in I^c) =  
\sum_{\gamma = \{\gamma(1),\ldots,\gamma(l)\} \in \Pi[p] \atop \gamma \, \text{ refines } \, \{I,I^c\}}\prod_{i=1}^{l}\mt^{(k_j: j \in \gamma(i))}(x_j ; j \in \gamma(i)), \]
and consequently we have that
\begin{eqnarray} 
\mt^{(k_1,\ldots,k_p)}(x_1,\ldots,x_p;n) & = & m^{(k_1,\ldots,k_p)}(x_1,\ldots,x_p;n) - m^{(k_j: j \in I)}(x_j ; j \in I)m^{(k_j: j \in I^c)}(x_j ; j \in I^c) \no \\
&  +  & \sum_{\gamma = \{\gamma(1),\ldots,\gamma(l)\} \in \Pi[p] \atop l > 1, \gamma \, \text{ mixes } \, \{I,I^c\}} \prod_{i=1}^{l}\mt^{(k_j: j \in \gamma(i))}(x_j ; j \in \gamma(i)).
\label{eqn:ursellreln}
\end{eqnarray}

\noindent Now, by induction one can show that for all $p \geq 1$ and for any configuration $x_1,\ldots,x_p \in V$, there exists a partition $I,I^c$ of $[p]$ with $d(\{x_j ; j \in I\},\{x_j ; j \in I^c\}) \geq diam(x_1,\ldots,x_p)/(p-1)$ where $diam$ is the graph diameter of the set and defined as $diam(x_1,\ldots,x_p) := \sup_{i \neq j} \|-x_i + x_j\|.$ Thus using \eqref{eqn:ursellreln}, the above partition and the clustering condition \eqref{eqn:clust_randfields}, we can inductively show that 
\begin{equation}
\label{eqn:clust_Ursell}
\sup_{1 \leq n \leq \infty} \sup_{x_1,\ldots,x_p \in W_n}|\mt^{(k_1,\ldots,k_p)}(x_1,\ldots,x_p;n)| \leq C^{\top}_p \phi(c^{\top}_pdiam(x_1,\ldots,x_p)), 
\end{equation}
where $C^{\top}_p, c^{\top}$ are finite constants. Now using \eqref{eqn:clust_Ursell}, we derive that
\begin{eqnarray}
\sup_{x_1 \in W_n} \sum_{x_2,\ldots,x_p \in W_n}|\mt^{(k_1,\ldots,k_p)}(x_1,\ldots,x_p;n)| 
& \leq & \sup_{x_1 \in W_n} C^{\top}_p \sum_{x_2,\ldots,x_p \in W_n} \prod_{i=2}^p\phi(c^{\top}_pd(x_1,x_i))^{1/(p-1)} ,\no \\
& = & \sup_{x_1 \in W_n} C^{\top}_p (\sum_{x \in W_n} \phi(c^{\top}_pd(x_1,x))^{1/(p-1)})^{p-1} < \infty,
\end{eqnarray}
where the finiteness is due to the summability of $\phi$ as in \eqref{eqn:phisum}. The above bound trivially implies \eqref{eqn:ursellbound} and thus we have completed the proof. 

\end{proof}

\subsection{\bf Proof of  the weak law of large numbers and the multivariate central limit theorem -Theorems \ref{thm:weakl} and \ref{thm:multimain}}
\label{sec:proofweak_clt}

Armed with the powerful clustering result for $H^{\xi}_n$, we can now give the proofs of the weak law of large numbers and the central limit theorems. 
\begin{proof}[Proof of Theorem \ref{thm:weakl}]
We will prove the results under Assumption 2. The proof can be trivially adapted under Assumption 1 by setting $\varphi(t) = 0$ for large enough $t$ and using summability of $\tilde{\phi}$ as in \eqref{eqn:phisum} due to Theorem \ref{thm:clustering_rand_measure}.
We shall first show that
\begin{equation}
\label{eqn:expconv}
 w_n^{-1}\sE(H^{\xi}_n) \to  m^{1}(O),
 \end{equation}
and this along with the variance asymptotics to be proved and Chebyshev's inequality suffices to prove the weak law of large numbers. 
By stationarity of $\P$, we have that $ \sE(\xi(x,\P)) =  \sE(\xi(O,\P))$ and hence
\begin{eqnarray}
 | w_n^{-1}\sE(H^{\xi}_n) - \sE(\xi(O,\P))|
& \leq & w_n^{-1} \sum_{x \in W_n} \sE( |\xi(x,\P_n) - \xi(x,\P)|\1[R(x,\P_n) \geq d(x,\partial W_n)]), \no \\
& \leq & 2w_n^{-1} (M_p)^{1/p} \sum_{x \in W_n}\sP(R(x,\P_n) \geq d(x,\partial W_n))^{1/q},
\end{eqnarray}
where in the last inequality we have used H\"{o}lder's inequality with $q \geq 1$ such that $\frac{1}{p} + \frac{1}{q} = 1$ ($p > 1$ is chosen as in the theorem assumption) and the moment condition \eqref{eqn:pmom}. By the property of radius of stabilisation \eqref{eqn:quasilocal}, we have that for any $r > 0$, 
\begin{eqnarray}
&& w_n^{-1} \sum_{x \in W_n}\sP(R(x,\P_n) \geq d(x,\partial W_n))^{1/q}  \leq w_n^{-1} A \sum_{x \in W_n} \varphi(n-|x|)^{1/q}   + w_n^{-1} |\partial W_n| \no \\
& \leq & w_n^{-1}|W_{n-r}| A \varphi(r)^{1/q} + A w_n^{-1} |W_n \setminus W_{n-r}|
\leq  w_n^{-1}|W_{n-r}| A \varphi(r)^{1/q} + A \sum_{j=0}^r w_{n-j}^{-1}|\partial W_{n-j}|.
\end{eqnarray}
%
Thus, by the $b$-amenability of $G$ we conclude that the second term above converges to $0$, and we have
\[ \limsup_{n \to \infty} w_n^{-1} \sum_{x \in W_n}\sP(R(x,\P_n) \geq d(x,\partial W_n))^{1/q} \leq A \varphi(r)^{1/q} ,\]
for any $r > 0$. Now letting $r \to \infty$, the proof of expectation asymptotics \eqref{eqn:expconv} is complete.

\noindent Now moving onto variance asymptotics, we have that
\[  \Var(H_n^{\xi}) =  \sum_{x, y \in W_n} \bigl( \sE(\xi(x,\P_n)\xi(y,\P_n)) - \sE(\xi(x,\P_n))\sE(\xi(y,\P_n)) \bigr). \]
For $x \in W_n$, we set $\P^x_n = \P \cap (W_n - x)$ and $W'_n(z) = W_n \cap (-z + W_n )^c$. Now, by change of variables, we have
\begin{eqnarray}
&& \Var(H_n^{\xi})  =  \sum_{x \in W_n, z \in W_n - x} \bigl( \sE(\xi(x,\P_n)\xi(z + x,\P_n)) - \sE(\xi(x,\P_n))\sE(\xi(z+x,\P_n)) \bigr), \no \\
& = & \sum_{z \in \mZ^d, x \in W_n}  [ m^{1,1}(0,z;\P^x_n) - m^1(O;\P^x_n)m^1(z;\P^x_n) ] \1[x \in (-z + W_n)] ,\no \\
& = & \sum_{z \in \mZ^d, x \in W_n}  [ m^{1,1}(0,z;\P^x_n) - m^1(O;\P^x_n)m^1(z;\P^x_n) ] 
-  \sum_{z,x \in \mZ^d}  [ m^{1,1}(0,z;\P^x_n) - m^1(O;\P^x_n)m^1(z;\P^x_n) ] \1[x \in W'_n(z)]. \no\\
\end{eqnarray}
By Theorem \ref{thm:clustering_rand_measure}, we have that
$$ w_n^{-1} \sum_{z,x \in \mZ^d}  | m^{1,1}(0,z;\P^x_n) - m^1(O;\P^x_n)m^1(z;\P^x_n) | \1[x \in W'_n(z)] \leq w_n^{-1} \tilde{C}_2\sum_{z \in \mZ^d} \tilde{\phi}(\tilde{c}_2|z|) |W'_n(z)|.$$
Since $\sum_{z \in \mZ^d} \tilde{\phi}(\tilde{c}_2|z|)$ is summable as $\tilde{\phi}$ is fast-decreasing, $w_n^{-1}|W'_n(z)| \leq 1$ and further, for all $z \in V$, by $b$-amenability of the discrete Cayley graph, we have that
$$w_n^{-1}|W'_n(z)|  \leq w_n^{-1} \sum_{j=n-|z|}^n |\partial W_j| \to 0.$$ 
So, we can use dominated convergence theorem to conclude that 
$$\sum_{z,x \in \mZ^d}  [ m^{1,1}(0,z;\P^x_n) - m^1(O;\P^x_n)m^1(z;\P^x_n) ] \1[x \in W'_n(z)] \to 0,$$
as $n \to \infty$ as required. Now, the proof is complete if we show that
\begin{equation}
\label{eqn:varconv1}
 w_n^{-1} \sum_{z \in \mZ^d, x \in W_n}  [ m^{1,1}(0,z;\P^x_n) - m^1(O;\P^x_n)m^1(z;\P^x_n)] \to  \sum_{z \in \mZ^d} (m^{1,1}(O,z) - m^{1}(O)^2). 
\end{equation}
The above convergence result follows if we show that
\[ \sum_{z \in \mZ^d} w_n^{-1} \sum_{x \in W_n} | m^{1,1}(0,z;\P^x_n) - m^1(O;\P^x_n)m^1(z;\P^x_n) - m^{1,1}(O,z) + m^{1}(O)^2 | \to 0 .\]
From Theorem \ref{thm:clustering_rand_measure}, we have that the modulus is bounded above by $2\tilde{C}_2\tilde{\phi}(\tilde{c}_2|z|)$ and hence we have that
\[  w_n^{-1} \sum_{x \in W_n} | m^{1,1}(0,z;\P^x_n) - m^1(O;\P^x_n)m^1(z;\P^x_n) - m^{1,1}(O,z) + m^{1}(O)^2 | \leq 2\tilde{C}_2\tilde{\phi}(\tilde{c}_2|z|),\]
and since the RHS is summable over $z \in V$, we can apply dominated convergence to conclude \eqref{eqn:varconv1} as required provided we prove that for all $z \in V$ as $n \to \infty$, 
\begin{equation}
\label{eqn:varconv2}
 w_n^{-1} \sum_{x \in W_n} | m^{1,1}(0,z;\P^x_n) - m^1(O;\P^x_n)m^1(z;\P^x_n) - m^{1,1}(O,z) + m^{1}(O)^2 | \to 0.
\end{equation} 
This can be shown term-wise. By using the arguments in the proof of expectation asymptotics \eqref{eqn:expconv}, we derive that 
\[  w_n^{-1} \sum_{x \in W_n} |m^1(O;\P^x_n)m^1(z;\P^x_n) - m^{1}(O)^2 | \leq 2w_n^{-1}M_1 \bigl( \sum_{x \in W_n} |m^1(O;\P^x_n) - m^{1}(O)| + |m^1(z;\P^x_n) - m^{1}(O)| \bigr).
\]
Now note that by Cauchy-Schwarz inequality and definition of radius of stabilisation, 
$$|m^1(z;\P^x_n) - m^{1}(z)| \leq 2M_2^{1/2} \sP(R(z,\P_n) \geq d(x,\partial W_n))^{1/2},$$
and similarly for the other term $|m^1(O;\P^x_n) - m^{1}(O)|.$ Now we can argue exactly like in the proof of \eqref{eqn:expconv} to obtain that 
\[  w_n^{-1} \sum_{x \in W_n} |m^1(O;\P^x_n)m^1(z;\P^x_n) - m^{1}(O)^2 | \to 0 .\]
Now, again using triangle inequality, we have that
\begin{eqnarray*}
 | \sE&&(\xi(O,\P^x_n)\xi(z,\P^x_n))  -   \sE(\xi(O,\P)\xi(z,\P))|\\ && \leq (\sE(\xi(O,\P^x_n)^2)^{1/2} + \sE(\xi(z,\P)^2)^{1/2}) \times 
  \bigl(  \sE(|\xi(O,\P^x_n) - \xi(O,\P)|^2)^{1/2} +   \sE(|\xi(z,\P^x_n) - \xi(z,\P)|^2)^{1/2} \bigr),  \\
 && \leq  2M_2^{1/2}  \bigl(  \sE(|\xi(O,\P^x_n) - \xi(O,\P)|^2)^{1/2} +   \sE(|\xi(z,\P^x_n) - \xi(z,\P)|^2)^{1/2} \bigr). 
\end{eqnarray*}
Now, choose $q$ such that $2/p + 1/q = 1$ for $p > 2$ as assumed in the Theorem \ref{thm:weakl}. Then we can derive that
\begin{eqnarray*}
&  & \sE\bigl(|\xi(O,\P^x_n) - \xi(O,\P)|^2\bigr) \leq  \sE\bigl(|\xi(O,\P^x_n) - \xi(O,\P)|^2\1[R(O,\P_n) \geq d(x,\partial W_n)]\bigr) \\
&  & \\
& \leq & \sE\bigl(\xi(O,\P^x_n)^2\1[R(O,\P_n) \geq d(x,\partial W_n)]\bigr) + \sE\bigl(\xi(O,\P)^2\1[R(O,\P_n) \geq d(x,\partial W_n)]\bigr) \\
&  & + 2\sE\bigl(|\xi(O,\P^x_n)\xi(O,\P)|\1[R(O,\P_n) \geq d(x,\partial W_n)]\bigr), \\
&  & \\
& \leq & \sE(|\xi(O,\P^x_n|^p)^{2/p}\sP(R(O,\P_n) \geq d(x,\partial W_n))^{1/q}  +  \sE(|\xi(O,\P)|^p)^{2/p}\sP(R(O,\P_n) \geq d(x,\partial W_n))^{1/q} \\
&  & + 2\sE(|\xi(O,\P^x_n|)|^p)^{1/p} \sE(|\xi(O,\P)|^p)^{1/p} \sP(R(O,\P_n) \geq d(x,\partial W_n))^{1/q},\\
& \leq & 4M_p^{2/p} \sP(R(O,\P_n) \geq d(x,\partial W_n))^{1/q},
\end{eqnarray*}
and similarly,
\[ \sE(|\xi(z,\P^x_n) - \xi(z,\P)|^2)  \leq  4M_p^{2/P} \sP(R(z,\P_n) \geq d(x,\partial W_n))^{1/q}. \]
Thus, we have that 
\begin{eqnarray*} 
 w_n^{-1} &\sum_{x \in W_n}& | m^{1,1}(0,z;\P^x_n) - m^{1,1}(O,z)| \\
& \leq &  4M_p^{(2+p)/2p} w_n^{-1} 
\times \sum_{x \in W_n} ( \sP(R(O,\P_n) \geq d(x,\partial W_n))^{1/2q}  + \sP(R(z,\P_n) \geq d(x,\partial W_n))^{1/2q} ) 
\,\,\,\,\,\,\,\,\,\,\,\to \,\,\,\,\, 0, 
\end{eqnarray*}
and this proves \eqref{eqn:varconv2} and consequently \eqref{eqn:varconv1} as well. This completes the proof of variance asymptotics.
\qed \end{proof}

\begin{proof}[Proof of Theorem \ref{thm:multimain}]
The proof of this theorem is via Cram\'er-Wold theorem. Let $N = (N_1,\ldots,N_k)$ be the Gaussian vector with mean zero and covariance matrix $\Sigma := (\Sigma_{i,j})$ as defined in the Theorem. By the Cram\'er--Wold theorem, it suffices to show that for all $t \in \mR^k$, we have
\begin{equation}
\label{eqn:cramerwold}
 \left\langle t, \left(\frac{\bar{H}_n - \EXP{\bar{H}_n}}{\sqrt{w_n}}\right) \right\rangle \stackrel{d}{\rightarrow} \langle t,N\rangle,
\end{equation}
where $\langle t,s\rangle := \sum_{i=1}^k t_i\,s_i$ stands for the usual dot product between two $k$-dimensional real vectors.  
\noindent Fix $t \in \mR^k$. For $x \in V$ and spin model $\mu$, set $\xi(x,\mu) := \sum_it_i\xi_i(x,\mu)$. Note that $\langle t,\bar{H}_n\rangle  = H^{\xi}_n.$ Further, we observe that $(\xi,\P)$ satisfies Assumption 1 (resp. Assumption 2) of Theorem \ref{thm:main} if all the $(\xi_i,\P), i = 1,\ldots,k$ satisfies Assumption 1 (resp. Assumption 2) of Theorem \ref{thm:main}. Further, as with the variance asymptotics in Theorem \ref{thm:weakl}, we can show that for all $ 1 \leq i,j \leq k$,
\[ w_n^{-1}\COV{H^{\xi_i}_n,H^{\xi_j}_n}  \to \Sigma_{i,j} .\]
One can also deduce the above asymptotic directly by noting that 
$$ \COV{H^{\xi_i}_n,H^{\xi_j}_n} = \frac{1}{2} \bigl( \Var(H^{\xi_i + \xi_j}_n) - \Var(H^{\xi_i}_n) - \Var(H^{\xi_j}_n) \bigr), $$
and the asymptotics for each of the three terms on the RHS follow from Theorem \ref{thm:weakl}. After appropriate cancellations, we obtain the desired covariance asymptotics. Thus, we derive that
\[ w_n^{-1} \Var(H^{\xi}_n) \to \sum_{1 \leq i,j \leq k} t_it_j \Sigma_{i,j} = \Var(\langle t,N\rangle). \]

\noindent If $\Var(\langle t,N\rangle) = 0$, then $\langle t,N\rangle = 0$ a.s. i.e., a degenerate random variable. Then, since $w_n^{-1}\Var(H^{\xi}_n) \to 0$, we derive \eqref{eqn:cramerwold} via Chebyshev's inequality trivially in this case. Alternatively, if $\Var(\langle t,N\rangle) > 0$, then since $(\xi,\P)$ satisfies either Assumptions 1 or 2 in Theorem \ref{thm:main} and we also have that $\Var(H^{\xi}_n) = \Omega(w_n)$, we can conclude \eqref{eqn:cramerwold} from Theorem \ref{thm:main} and thereby completing the proof of our multivariate central limit theorem. 
\end{proof}

\subsection{\bf Proof of Theorems \ref{thm:clt-alpha-local} and \ref{thm:clt-alpha-quasi-local}}
\label{sec:proof-mixing-clt}

Before we start documenting the proofs of related to central limit theorems for sums of score functions
of random fields satisfying certain mixing conditions, we remind the reader that we have so far stayed clear
of an important aspect, that of non-degeneracy of the limiting distribution. We refer the reader to \cite{Doukhan},
and references therein, for a detailed discussion on various conditions / assumptions which ensure positivity
of the limiting variance. We shall not delve any deeper into this very delicate issue of the non-degeneracy
of the limiting distribution.

\begin{proof}[Proof of Theorem \ref{thm:clt-alpha-local}]
We begin the proof by first noting that we shall follow the proof of similar limit theorem
mentioned in \cite[Theorem 1.1]{Bradley15}. Observe that the strong $\alpha$-mixing condition as stated in
\cite{Bradley15} is weaker than the strong $\alpha$-mixing we defined in Section \ref{sec:mixing}. Therefore,
if $\X$ satisfies the strong $\alpha$-mixing condition stated in our Section \ref{sec:mixing}, then $\X$ also
satisfies the strong $\alpha$-mixing condition stated in \cite{Bradley15}.

\noindent Consider the total mass,
\begin{equation}
\label{eqn:Hnm}
H^{\xi}_n =  \sum_{x \in W_n} \xi(x,\X_n).
\end{equation} 
Note that since $\xi$ is a local functional of the point process with with finite radius, say $r$, the random variables 
$\xi(x,\X_n)$ are not identically distributed for $x\in \X_n$. However, as observed earlier, 
the random field $\{\xi(x,\X)\}_{x\in G}$ is indeed stationary. We, therefore shall prove the invariance
principle for $\widetilde{H}^{\xi}_n = \sum_{x \in W_n} \xi(x,\X)$, and show that the difference 
$(H^{\xi}_n-\widetilde{H}^{\xi}_n)$ is negligible under the volume scaling.
We shall break the proof into two steps: first we shall prove that $\E\left[(H^{\xi}_n-\widetilde{H}^{\xi}_n)^2\right]$
is $\textrm{o}(w_n)$, thereafter we shall prove the invariance principle for $\widetilde{H}^{\xi}_n$.

\noindent\textsc{Step 1:} Let $r$ (fixed) be the range of the local functional $\xi$, we have
\begin{equation}
\left( H^{\xi}_n-\widetilde{H}^{\xi}_n\right) = \sum_{x\in W_n\setminus W_{n-r}} \left[ \xi(x,\X_n) - \xi(x,\X)\right].
\end{equation}
Writing $\zeta_{x,n} = \left( \xi(x,\X_n) - \xi(x,\X)\right)$, 
consider the triangular array $\{\zeta_{x,n}:\, x\in W_n\setminus W_{n-r},\,\,\, n\ge 1\}$.
Clearly, $\zeta_{x,n}$ is another local functional with range $r$. Therefore, 
$\COV{\zeta_{x,n},\zeta_{y,n}} \le \rho(|x-y|+2r)$. Now let $j$ be the index for which
$\rho(j) < 1$, then following the same arguments as set forth in \cite{Bradley15}, we conclude that
$$\E\left( H^{\xi}_n-\widetilde{H}^{\xi}_n\right)^2 \le C \sum_{x\in W_n\setminus W_{n-r}} \E\left( \xi(x,\X_n) - \xi(x,\X)\right)^2, $$
where the constant $C$ is given by $2\,(j-2r)^d /\left( 1- \rho(j) \right)$. Using now the regularity assumption concerning
the moments of the local score function $\xi$, we conclude that,
$$\E\left( H^{\xi}_n-\widetilde{H}^{\xi}_n\right)^2 \le C_1 \left|W_n\setminus W_{n-r}\right|.$$
Next, using the estimate obtained earlier in the proof of Theorem \ref{thm:weakl}, we obtain the desired conclusion 
$$\E\left( H^{\xi}_n-\widetilde{H}^{\xi}_n\right)^2 = o(w_n).$$

\noindent\textsc{Step 2:} We shall now move to the second step of our proof, which involves the invariance principle
\begin{equation}\label{eqn:mixing-clt-tildeH}
\frac{\widetilde{H}^{\xi}_n - \E\left( \widetilde{H}^{\xi}_n\right)}{\sqrt{w_n}} \stackrel{d}{\Rightarrow} N(0,\sigma^2).
\end{equation}
First we note that the Lindeberg condition for $\X$ is automatically
satisfied due to the stationarity of the field $\X$, and by the same argument, the random field $\xi$
also does satisfy the Lindeberg condition. Then, we observe that if $\X$ is an $\alpha$-mixing
random field then so is $\xi(\cdot,\X)$, which in turn implies that $\xi$ satisfies the mixing condition 
stated in \cite[Theorem 1.1]{Bradley15}. More precisely,
using the same notation as set forth earlier, and writing $\alpha$ and $\alpha_{\xi}$ for the mixing coefficients 
corresponding to the field $\X$ and $\xi$, respectively, we have
$\alpha_{\xi}(s) \le \alpha(s+2r),$
implying that the random field $\xi$ is $\alpha$-mixing if the original field $\X$ is $\alpha$-mixing.
Now we invoke Theorem 1.1 of \cite{Bradley15} to conclude the statement \eqref{eqn:mixing-clt-tildeH},
which in turn, concludes the proof.
\qed
\end{proof}
Next, in order to prove Theorem \ref{thm:clt-alpha-quasi-local}, we shall use the following technical observation.

\begin{proposition}
\label{thm:clustering_mixing_rand_measure}
Let $G$ be a discrete Cayley graph and together with $(\xi,\X)$ such that
$G$ has polynomial growth as in Definition \ref{def:amen}, $\X$ is an exponentially mixing spin model, 
$\xi$ is a exponentially quasi-local score function as in \eqref{eqn:quasilocal} 
satisfying 
the $p$-moment condition \eqref{eqn:pmom} for all $p \geq 1$. 
Then, the random field $\{\xi(x,\X_n)\}_{x \in W_n}, 1 \leq n \leq \infty$ satisfies clustering in terms of mixed moments, i.e., there exists constants $\tilde{C}_K, \tilde{c}_K$ such that for all $x_1,\ldots,x_{p+q}$ with $s =d(\{x_1,\ldots,x_p\},\{x_{p+1},\ldots,x_{p+q}\})$ and $K = \sum_{i=1}^{p+q}k_i, k_i \geq 1$ for $i = 1,\ldots, p+q$,  such that writing
$$m^{k_1,\ldots ,k_p}(x_1,\ldots ,x_p) := \sE\left[\prod_{i=1}^p\xi(x_i,\X_n)^{k_i}\right],$$
we have that
\begin{equation}
\label{eqn:clust_rm}
| m^{k_1,\ldots,k_{p+q}}_n(z_1,\ldots,z_{p+q}) -  m^{k_1,\ldots,k_p}_n(z_1,\ldots,z_p) m^{k_{p+1},\ldots,k_{p+q}}_n(z_{p+1},\ldots,z_{p+q}) | \leq \tilde{C}_K \tilde{\phi}(\tilde{c}_K s),
\end{equation}
for a fast decreasing function $\tilde{\phi}$. Further, under Assumption 1 if $\phi$ is summable as in \eqref{eqn:phisum}, so is $\tilde{\phi}$. 
\end{proposition}

\begin{proof}
%
%
%
%
%

We shall now try to mimic the arguments used in the proof of Theorem \ref{thm:clustering_rand_measure}. Let us fix $x_1,\ldots,x_{p+q}$ and assume that $s = d(\{x_1,\ldots,x_p\},\{x_{p+1},\ldots,x_{p+q}\}).$ 
Further, without loss of generality, let $s > 4$ and $n \ge 10s$.  Set $t = t(s) :=  (s/4)^{\gamma}$ for a 
$\gamma \in (0,1)$ to be chosen later. Define for all $x_i$
 \[ \txi(x_i, \X_n)  = \xi(x,\X_n) \1 [R(x_i,\X_n) \leq t]. \]
Clearly, $\txi$ is a local statistic with radius of stabilisation $\tilde{R}(x,\X_n) \leq t$. Further set $\tilde{m}^{k_1,\ldots ,k_p}(x_1,\ldots ,x_p;n) := \sE[\prod_{i=1}^p\txi(x_i,\X_n)^{k_i}]$. Now, proceeding as in the proof of \eqref{eqn:bd_mm_truncation} and again using H\"{o}lder's inequality, the moment condition \eqref{eqn:pmom} and exponentially quasi-locality \eqref{eqn:quasilocal}, we have that
\begin{eqnarray}\label{eqn:moment-compare-2}
&  & | m^{k_1,\ldots,k_{p+q}}(x_1,\ldots, x_{p+q};n) -  m^{k_1,\ldots,k_p}(x_1,\ldots ,x_p;n) m^{k_{p+1},\ldots,k_{p+q}}(x_{p+1},\ldots, x_{p+q};n) | \nonumber \\
& \leq & 5M_K^{2K/(K +1)} (AK\varphi(t))^{1/(K+1)}  \nonumber\\
&&  + | \tilde{m}^{k_1,\ldots,k_{p+q}}(x_1,\ldots, x_{p+q}) -  \tilde{m}^{k_1,\ldots,k_p}(x_1,\ldots ,x_p;n) \tilde{m}^{k_{p+1},\ldots,k_{p+q}}(x_{p+1},\ldots, x_{p+q};n) | .
\end{eqnarray}%
Note, however, by the choice of $t$ and $\varphi$, we can replace $(AK\varphi(t))^{1/(K+1)}$ by an appropriate 
fast decaying functional $\tilde{\phi}_1(c_K\,s)$.
Then, using Theorem 3 in Section 1.2.2 of \cite{Doukhan} for the local statistics $\txi(x_i,\X_n)$, we obtain
\begin{eqnarray}
&& | \tilde{m}^{k_1,\ldots,k_{p+q}}_n(x_1,\ldots,x_{p+q}) 
-  \tilde{m}^{k_1,\ldots,k_p}_n(x_1,\ldots,x_p) \tilde{m}^{k_{p+1},\ldots,k_{p+q}}_n(x_{p+1},\ldots,x_{p+q}) | \no \\
&\le & 2 \left[\alpha(s-2t)\right]^{1/3} \left[\E\left( \prod_{i=1}^p \left| \xi^{k_{i}}(x_{i},\X_n)\right|\right)^3\right]^{1/3}
\left[\E\left( \prod_{j=1}^q \left| \xi^{k_{p+j}}(x_{p+j},\X_n)\right|\right)^3\right]^{1/3}. \no
\end{eqnarray}
Applying \eqref{eqn:pmom} and using H\"older's inequality, we can bound 
$$
\left[\E\left( \prod_{i=1}^p \left| \xi^{k_{i}}(x_{i},\X_n)\right|\right)^3\right]^{1/3}
\left[\E\left( \prod_{j=1}^q \left| \xi^{k_{p+j}}(x_{p+j},\X_n)\right|\right)^3\right]^{1/3} \le \tilde{C}_{1,K}.
$$
Finally, setting $\tilde{\phi}_2(s) = \left[\alpha(s-2t)\right]^{1/3}$, we have
\begin{equation} \label{eqn:moment-compare-3}
| \tilde{m}^{k_1,\ldots,k_{p+q}}_n(z_1,\ldots,z_{p+q}) 
-  \tilde{m}^{k_1,\ldots,k_p}_n(z_1,\ldots,z_p) \tilde{m}^{k_{p+1},\ldots,k_{p+q}}_n(z_{p+1},\ldots,z_{p+q}) |
\le \tilde{C}_{1,K} \, \tilde{\phi}_2(s). 
\end{equation}
Collating \eqref{eqn:moment-compare-2} and \eqref{eqn:moment-compare-3},
now we set $\tilde{\phi}(c_K\,s) = \max(\tilde{\phi}_1(c_K\,s),\tilde{\phi}_2(s))$, and 
$\tilde{C}_K = 5M_K^{2K/(K +1)} + \tilde{C}_{1,K} $.
Finally, using the exponential decay of the $\alpha$-mixing coefficient, we conclude this theorem.
\qed
\end{proof}


\begin{proof}[Proof of Theorem \ref{thm:clt-alpha-quasi-local}]
The proof follows from Proposition \ref{thm:clustering_mixing_rand_measure} and Theorem \ref{thm:cltrandfields}. 
\qed
\end{proof}

\remove{ \section{Extensions :}
\label{sec:extensions}

\begin{enumerate}
\item Scaling limits as random fields.

\item Other Euclidean lattices. 

\item Amenable and countable groups. 

\item A more complex spin model indexed by vertices, edges, faces et al...

\item Extension to random fields or equivalently $\mR$-valued or even finite-valued spin models. 

\item \cite[Lemma 4.2]{Grote16} gives conditions on cumulant bounds to obtain rates of normal approximation as well as moderate deviations. This is worth investigation. 
\end{enumerate}
}

\section{Appendix: Void probability estimate}
\label{sec:voidprob}

\begin{lemma} 
\label{lem:voidprob}
Let $\P$ be an exponential clustering spin model as in \eqref{eqn:expclust}. Then, there exists constants $A,a', \nu \in (0,\infty)$ such that for any $t > 0$, 
\begin{equation}
\label{eqn:voidprob}
\sup_{n \geq 1} \sup_{z \in W_n} \sP(\P_n(W_t(z)) = 0) \leq A e^{-a' t^{\nu}} .
\end{equation}
\end{lemma}
\begin{proof}
Without loss of generality, we shall assume that $c = 1$ in \eqref{eqn:expclust}. Define the random field $Z_x = \1[x \in \P]$ for $x \in V$. Thus, $\sP(\P_n(S) =  0) = \sE(\prod_{x \in S \cap W_n}(1 - Z_x)).$ Firstly, let $P,Q \subset V$ such that $s := d(P,Q)$ and $|P| = p, |Q| = q$. Then we have by the clustering property of the spin model that
\begin{eqnarray}
 \bigg|\sE(\prod_{x \in P \cup Q}(1-Z_x)) - \sE(\prod_{x \in P}(1-Z_x))\sE(\prod_{x \in Q}(1-Z_x))\bigg| 
&=&  
\bigg|\sum_{\substack{A \subset P, B \subset Q,\\ A \cup B \neq \emptyset}} (-1)^{|A|+|B|}  \bigl( \sE(\prod_{x \in A \cup B}Z_x) - \sE(\prod_{x \in A}Z_x)\sE(\prod_{x \in B}Z_x) \bigr)\bigg| ,\no \\
 \label{eqn:clust0spin} & \leq & \sum_{\substack{A \subset P, B \subset Q,\\ A \cup B \neq \emptyset}} C_{|A| + |B|}\phi(c_{|A|+|B|}s) \, \, \leq \, \,  2^{p+q} C_{p+q}\phi(s). 
\end{eqnarray}
Further, for any $S \subset V$ such that $s := s(S) = \min_{x \neq y \in S}|-y+x|$, $|S| = l$. 
Using the triangle inequality recursively,
we can conclude by induction that
\[ |\sE(\prod_{x \in S}(1-Z_x)) - \prod_{x \in S}\sE(1-Z_x)| \leq (l-1)2^lC_l\phi(s) .\]
Now let $a_* \in (0,\infty]$ be such that $\sE(Z_x) = \rho^{(1)}(O) =  1 - e^{-a_*}$.  Then combining with the above inequalities and that $C_l \leq C_*l^{al}$ by exponential clustering assumption, we have that for $S$ as above,
\begin{equation}
\label{eqn:voidprobbound}
\sP(\P(S) =  0) \leq  C_*(l-1)(2l^a)^l\phi(s) + e^{-la_*}.
\end{equation}
For any $t$ large and any $z \in W_n$ for $n > 10t$, we can find a subset $S' \subset W_t(z) \cap W_n$ such that $|S'| \in [A' t^{\alpha'},A' t^{\alpha'}+2]$ and $s(S') = w_*t^{\beta}$ for $0 < \alpha', \beta < \infty$. Such a subset exists because for all $z \in W_n$, there is at least a (non-intersecting) path of length $t$ in $W_t(z) \cap W_n$. Hence such a choice of $S'$ can be made with $\alpha, \beta = 1/2$. Then, we further choose a subset $S \subset S'$ such that $|S| \in [A_*t^{\alpha},A_*t^{\alpha}+2]$ where $2\alpha < b\beta$ with $b$ as in \eqref{eqn:expclust}. As $s(S) \geq s(S')$, we can derive from \eqref{eqn:voidprobbound} that
\[ \sup_{z \in W_n} \sP(\P_n(W_t(z)) = 0) \leq C_{**}t^{\alpha}(2t^{a\alpha})^{A_*t^{\alpha}}e^{-(w_*)^bt^{b\beta}} + e^{-a_*t^{\alpha}} ,\]
where $C_{**}$ is a product of the various constants involved. Thus $\P_n$ satisfies\eqref{eqn:voidprob} for all large enough $n$. 
\end{proof}

\section*{Acknowledgements}
DY is thankful for the discussions with Matthew Wright which led to his interest in this question and especially the applications to random cubical complexes.  The authors are also thankful to numerous comments by anonymous referees that has lead to an improved presentation. 

\bibliographystyle{spmpsci} 
\bibliography{cltspinmodels}

\end{document}